\RequirePackage{fix-cm}
\documentclass[smallcondensed, numbook]{svjour3}

\usepackage{graphicx}
\usepackage{amsfonts,amssymb,amsmath}
\usepackage[hidelinks]{hyperref}
\usepackage{cleveref}
\usepackage{tikz-cd} 
\usepackage{enumitem}
\usepackage[misc,geometry]{ifsym}
\usepackage{epstopdf}
\usepackage{bbm}

\smartqed

\journalname{}

\title{Functional Analysis and Exterior Calculus on Mixed-Dimensional Geometries\thanks{This research was supported in part by the Norwegian Research Council
grants: 233736, 250223.}}

\author{Wietse M. Boon \and Jan M. Nordbotten \and Jon E. Vatne}

\institute{Wietse M. Boon \at
	Department of Mathematics, KTH Royal Institute of Technology, Stockholm, Sweden \\(wietse@kth.se).
\and Jan M. Nordbotten \at
	Department of Mathematics, University of Bergen, Bergen, Norway. \\
	Department of Civil and Environmental Engineering, Princeton University, Princeton, USA.
\and Jon E. Vatne \at
	Department of Computer Science, Electrical Engineering and Mathematical Sciences, \\
	Western Norway University of Applied Sciences, Bergen, Norway
  }

\date{Received: date / Accepted: date}

\usepackage{amsopn}
\newcommand\myd{\mathrm{d}}

\newcommand{\Bbbd}{\mathbbm{d}}
\DeclareMathOperator{\Tr}{Tr}
\DeclareMathOperator{\mystar}{\star}
\DeclareMathOperator*{\arginf}{\arg\,\inf}
\newcommand{\tx}[1]{\makebox[\widthof{$C^\infty\Lambda^n(\overline{X})$}]{$#1$}} 
\newcommand{\hatj}{\hat{\jmath}}

\begin{document}

\maketitle

\begin{abstract}
We are interested in differential forms on mixed-dimensional geometries, in the sense of a domain containing sets of $d$-dimensional manifolds, structured hierarchically so that each $d$-dimensional manifold is contained in the boundary of one or more $d + 1$ dimensional manifolds.

On any given $d$-dimensional manifold, we then consider differential operators tangent to the manifold as well as discrete differential operators (jumps) normal to the manifold. The combined action of these operators leads to the notion of a semi-discrete differential operator coupling manifolds of different dimensions. We refer to the resulting systems of equations as mixed-dimensional, which have become a popular modeling technique for physical applications including fractured and composite materials. 

We establish analytical tools in the mixed-dimensional setting, including suitable inner products, differential and codifferential operators, Poincar\'e lemma, and Poincar\'e--Friedrichs inequality. The manuscript is concluded by defining the mixed-dimensional minimization problem corresponding to the Hodge-Laplacian, and we show that this minimization problem is well-posed.

\keywords{mixed-dimensional differential operators, mixed-dimensional geometry, exterior calculus.}

\end{abstract}

\section{Introduction}\label{introduction}

Partial differential equations of reduced dimension are common in
mathematical modeling, and examples include shells, membranes,
fractures, geological formations, etc.~(see e.g. \cite{1,2,3,4}). When such structures
are embedded into a higher-dimensional surrounding media, it is often advantageous to 
consider the resulting problem as a mixed-dimensional problem (also
referred to as hybrid-dimensional by some authors). To our knowledge,
such problems have until present been considered in a case by case basis,
based on the needs of the various applications. In this paper, we treat
the case of hierarchical co-dimension one, that is to say, for an
$n$-dimensional domain, we consider manifolds of dimension
$n-1$, and their intersections, again manifolds, of dimension $n-2$, $n-3$ and
so on. We refer to this geometric construction as mixed-dimensional.

In order to provide a unified theoretical framework for
mixed-dimensional partial differential equations, we use the setting of
exterior calculus, identify a suitable notion of spaces of
alternating $k$-forms for the mixed-dimensional geometry, and equip
the spaces with inner products and norms. We define a
discrete-continuous differential operator acting on alternating
$k$-forms, and its adjoint, a codifferential operator. We show that our spaces and
differential operator form a de Rham complex, with the same cohomology
spaces as the full domain.

In the application section, we show that quadratic minimization problems defined in terms of our
mixed-dimensional differential operator are well-posed. It follows that
we can write variational forms of partial differential equations,
formally consistent with the strong form of the partial differential
equations.

We note several relationships between our current work and previous literature. Firstly, in our construction, lower-dimensional manifolds are restricted such that they coincide with boundaries of manifolds of one higher (topological) dimension, and their boundaries coincide with manifolds of one lower dimension. As such, the connections between the manifolds are similar to those appearing in \v{C}ech-complexes, and we arrive at double-complexes that have much in common to the \v{C}ech-de Rham complexes \cite{6,Weil1952}. However, the details of our construction differ, as we are considering not overlaps of open sets, but actual lower-dimensional manifolds. There is therefore also a parallel to singular cohomology as presented in \cite{6}. Another connection can be found to distributional forms, and in particular the refinement of those developed by Licht \cite{11}. Our work can be seen as generalizing that work from finite-dimensional spaces on simplicial complexes (motivated by polynomial forms and finite element discretizations), to infinite-dimensional spaces (continuous functions and Sobolev spaces) on more general manifolds, not necessarily simplicial. This generalization is essential for our goal of constructing mixed-dimensional partial differential equations corresponding to those found in the cited applications. On the other hand, our constructions are more concrete than the abstract setting of distributional differential forms presented by Melrose \cite{melrose2011remark}, where there is no restriction on the locus where a distributional form fails to be represented by ordinary differential forms. This allows us to not only consider Sobolev spaces in the sense of distributional forms, but also explicitly represent these spaces in terms of classical Sobolev spaces on the separate manifolds. The analysis conducted herein provides the theoretical foundation for the research program outlined by the authors for
modeling, analysis and numerical discretization of mixed-dimensional physical problems \cite{5}.

The key novel contributions of this work are as follows: 
\begin{enumerate}[label=\arabic*)]
	\item The use of directed acyclic graphs (DAGs) to represent connections between manifolds comprising a mixed-dimensional geometry (\Cref{def:forest}).
	\item Explicit representation of Sobolev spaces for mixed-dimensional geometries in terms of classical Sobolev spaces on individual manifolds (\Cref{H-char}). 
	\item A cohomology theory extending discrete distributional forms to distributional differential forms which are locally continuous with respect to mixed-dimensional geometries (\Cref{Thm: Cohomology H}). 
	\item An explicit link between our mixed-dimensional construction (and thus implicitly also \v{C}ech-de Rham complexes and discrete distributional forms) and current research within analysis of physical models for mixed-dimensional partial differential equations in applications (\Cref{sec4}). 
\end{enumerate}

\section{Background and notation}\label{sec2}

Herein we summarize the geometric setting of interest, and the
notational conventions used. 

\subsection{Geometric setting}\label{sec2.1}
The precise description of the geometric setting is quite technical, but is required to make the calculations later in the manuscript precise. An intuitive understanding of the geometry can be obtained by considering the examples and figures.

We consider an open domain $Y \subset \mathbb{R}^n$, together with its boundary $\partial Y$, which is decomposed into two parts $\partial_N Y$ and $\partial_D Y$ such that $\partial Y = \partial_N Y \cup \partial_D Y $ and $\partial_N Y \cap \partial_D Y = \emptyset$.  We are interested in partitions into disjoint, connected manifolds $\Omega_i^d \subset Y_N$, of the domain $Y_N = Y\cup \partial_N Y$. The partition is indexed by $i\in I$ with $d_i$ the dimension, and we denote the set of indexes for which $d_i = d$ as $I^d$. We require that the manifolds $\Omega_i^{d_i}$ are diffeomorphic to bounded open sets $X_i \subset \mathbb{R}^{d_i}$. Furthermore, we require that $\overline{X}_i$ is a $d_i$-manifold with boundary $\overline{X}_i \setminus X_i$. In the special case that $X_i$ is the $d_i$-dimensional unit ball, we write $X_i = B^{d_i}$.
We require that $\cup_{i\in I} \Omega_i^{d_i} = Y_N$ disjointly, i.e that every point in $Y_N$ can be associated with a unique $\Omega_i^{d_i}$. 
We frequently omit the superscript $d_i$, i.e. write $\Omega_i = \Omega_i^{d_i}$, and refer to the collection of manifolds simply as $\Omega$.

In order to precisely discuss this geometry and the restrictions we require, we will represent it as rooted directed acyclic graphs (which we will refer to as DAGs) with coordinate maps. Thus for every domain $\Omega_i$ with index $i\in I$, let $\mathfrak{S}_i$ be the DAG associated with $i$ of maximal depth $d_i$. We will choose a global enumeration of nodes such that the root of DAG $\mathfrak{S}_i$ is node $i$. Within each DAG, we denote the descendants of any node $j$ as $I_j$. Moreover, we define the subset $I_i^{d} \subset I_i$ as all descendants $l \in I_i$ with associated dimension $d_l = d$.

In a slight abuse of terminology, we will refer to the family of all the DAGs as the forest $\mathfrak{F}$, defined as
\begin{equation}
	\mathfrak{F} = \bigcup_{i\in I} \mathfrak{S}_i.
	\label{eq2.1.2b}
\end{equation}
While we will refer to $\mathfrak{F}$ as a forest, we note that it can be represented by a star graph with each leaf corresponding to the root of DAG $\mathfrak{S}_i$ with $i \in I$. When using the interpretation of the forest as a star graph, we will use the index $0$ for the (global) root.
In order to simplify summations, we slightly abuse notation by using DAGs and the forest also as index sets, e.g. write $j \in \mathfrak{F}$ to denote the indexes of all nodes in the forest $\mathfrak{F}$.

We require that each DAG $\mathfrak{S}_i$ is endowed with a family of manifolds and maps, such that for every node $j \in I_i $, there exists an orientable manifold $X_j$ and a smooth bijective coordinate map $\phi_{i,j} : \overline{X}_j \rightarrow \overline{\partial_j X_i}$, where $\partial_j X_i\subset \partial X_i$ is the image of $\phi_{i,j}(X_j)$. For the root of the DAG $\mathfrak{S}_i$, the unique parent is the global root $0$, for which we only require that the mapping is surjective, and use the convention that $\partial_i X_0=\Omega_i$, and thus $\phi_{0,i} : \overline{X}_i \rightarrow \overline{\Omega}_i$. We require that all maps $\phi_{i,j}$ are diffeomorphisms on $X_j$.

\begin{definition}\label{def:DAG}
	A rooted DAG $\mathfrak{S}_i$ with $i\in I$, is \emph{conforming} to $\Omega_i$ if for all nodes $j \in \mathfrak{S}_i$:
\begin{enumerate} 
	\item There exists a root $s_j\in I$ such that $\phi_{0,j}(X_j) = \Omega_{s_j}$. Moreover, $s_i = i$ for the root $i$.
	\item Compound maps telescope in the sense that for every $l \in I_i$ with $j \in I_l$, it holds that
	$\phi_{i,j} = \phi_{i,l} \circ \phi_{l,j}$.   
	\item The descendants uniquely cover the parent node in the sense that \\
	$\bigcup_{j\in \mathfrak{S}_i}\phi_{i,j}(X_j) = \overline{X}_i \setminus  \phi_{0,i}^{-1}(\partial \Omega_i \cap \partial_D Y).$ \\
	In other words, each point $x_i$ in reference domain $X_i$ and on its boundary is uniquely associated to a node $j \in \mathfrak{S}_i$ and a point $x_j \in X_j$ such that $x_i = \phi_{i, j}(x_j)$. For $x_i$ on the boundary $\partial X_i$, we have $j \in I_i$ a descendant of $i$ whereas for $x_i$ in the interior of $X_i$, we have $j = i$. All points that are mapped to the physical boundary $\partial_D Y$ by $\phi_{0, i}$ are exempt from this rule.
\end{enumerate}
\end{definition}

Note that we do not require that $s_{j_1} \ne s_{j_2}$ for $j_1, j_2 \in I_i$, such that $\phi_{0, i}$ may map several sections of the boundary $\partial X_i$ to the same $\Omega_j$. This allows for the manifolds $X_i$ to be mapped to manifolds with various kinds of ``loose ends'' or ``slits'', as seen in the examples below.  Moreover, it is a consequence of the definition of a conforming DAG that for any $i\in I$ and $j\in I_i$, it holds that $I_j \subset I_i$. 

\begin{definition}\label{def:forest}
A forest $\mathfrak{F}$ is conforming to $\Omega$ if the DAGs $\mathfrak{S}_i$ are conforming to $\Omega_i$ for all $i\in I$ in the sense of Definition \ref{def:DAG}. The forest is open if $\partial_D Y = \partial Y$, and is closed if $\partial_N Y =\partial Y$. 
\end{definition}

Open forests allow for boundary conditions on $\partial Y$ to be applied to the manifolds $\Omega_i$ which extend to the boundary. Conversely, closed forests allow for differential equations to be defined on the boundary. Here, and in the following, we will only consider partitions $\Omega$ for which there exists a conforming forest.  

\begin{example}
A domain partitioned by a simplicial complex has a conforming forest.
\end{example}

\begin{example}
An allowable 2D partitioning which is not a simplicial complex is given \Cref{fig: slit}. Also in the figure is the map of the largest top-dimensional domain, as well as the full forest $\mathfrak{F}$. We note that since $X_4$ is not isomorphic to a ball, this partition is also not a cellular complex.
\end{example}

\begin{figure}[tbhp!]
\centering
\includegraphics[width=\textwidth]{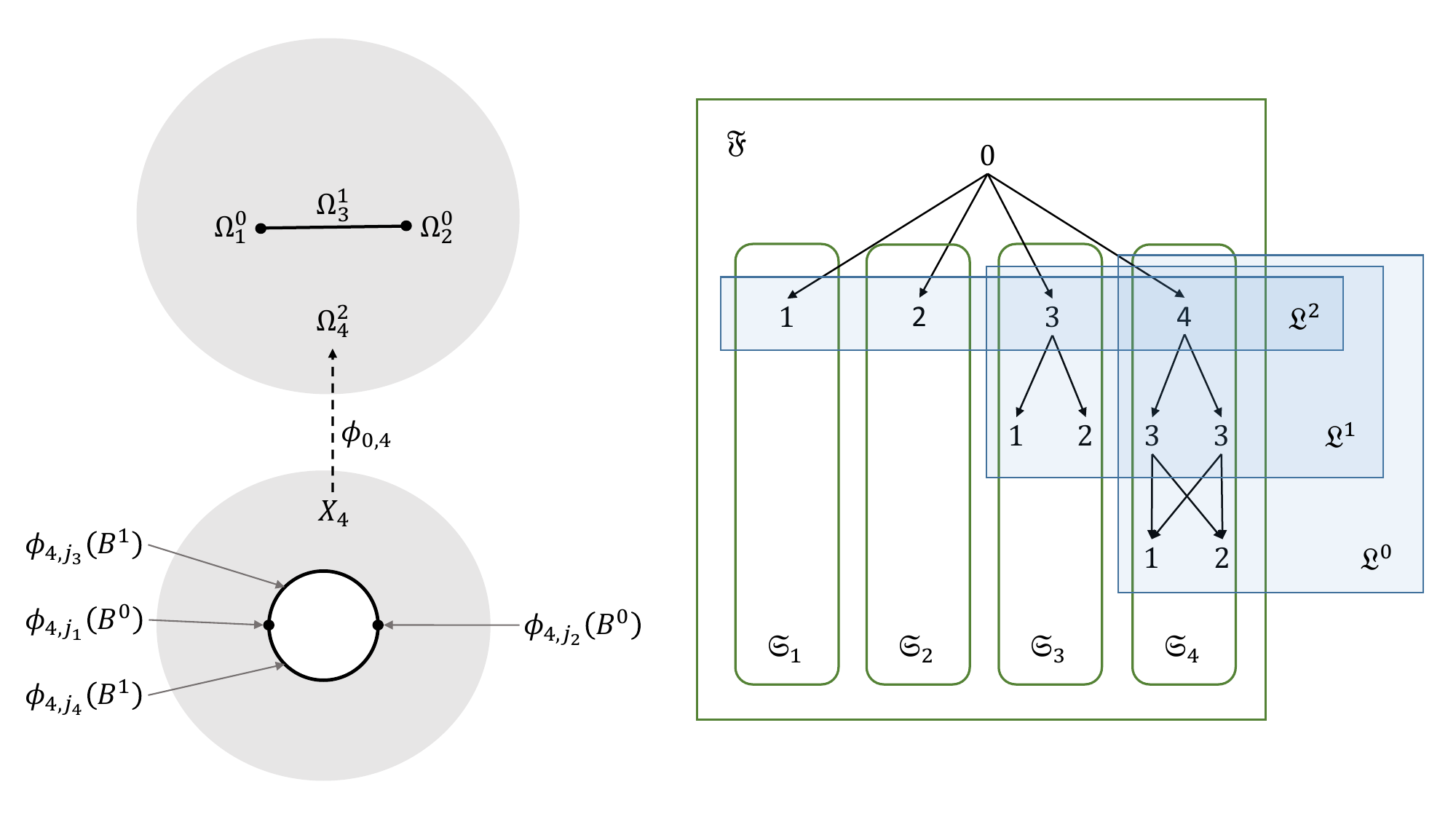}
\caption{Example geometry on an open domain in 2D (upper left); the structure of the reference domain $X_4$ and the map to $\Omega_4^2$ (lower left); and the corresponding forest $\mathfrak{F}$ (right). In this example, there are four manifolds (and DAGs), numbered such that $I^0=\{1,2\}$, $I^1=\{3\}$ and $I^2=\{4\}$. The nodal labels in the DAGs correspond to the identifications of domains $s_j$ - the global numbering of nodes is not shown. Note in particular that due to the ``slit'', $\Omega_3^1$ corresponds to two parts of the boundary of $\Omega_4^2$, denoted by $j_3$ and $j_4$ in the reference domain. Thus, both $\phi_{0, j_3}(X_{j_3}) = \Omega_3^1$ and $\phi_{0, j_4}(X_{j_4})=\Omega_3^1$ and this is reflected in the DAG $\mathfrak{S}_4$ by the double occurrence of the label $3$. We have also highlighted groups of nodes whereon function spaces will be defined in later sections.  
}\label{fig: slit}
\end{figure}

\begin{example}
	\Cref{fig: intersect} shows a three-dimensional partition in which two of the manifolds ($\Omega_9^3$ and $\Omega_8^2$) are not isomorphic to a ball. 
	As shown in the right of the figure, the embedded network corresponds to a three-dimensional hole in reference space $X_9$.
\end{example}

\begin{remark}
The previous example illustrates a limitation of our construction, as we have required the manifolds $\Omega_i^{d_i}$ to be diffeomorphic to bounded open sets in $\mathbb{R}^{d_i}$. Thus, with reference to Figure \ref{fig: intersect}, the 1-manifold $\Omega_6^1$ cannot be a closed curve, and the 0-manifold $\Omega_3^0$ is introduced. We emphasize that this limitation is for notational convenience, in order to be able to refer to a unique coordinate map for each manifold, and that our results extend to manifolds with local coordinate maps. 
\end{remark}

\begin{figure}[tbhp!]
\centering
\includegraphics[width=.9\textwidth]{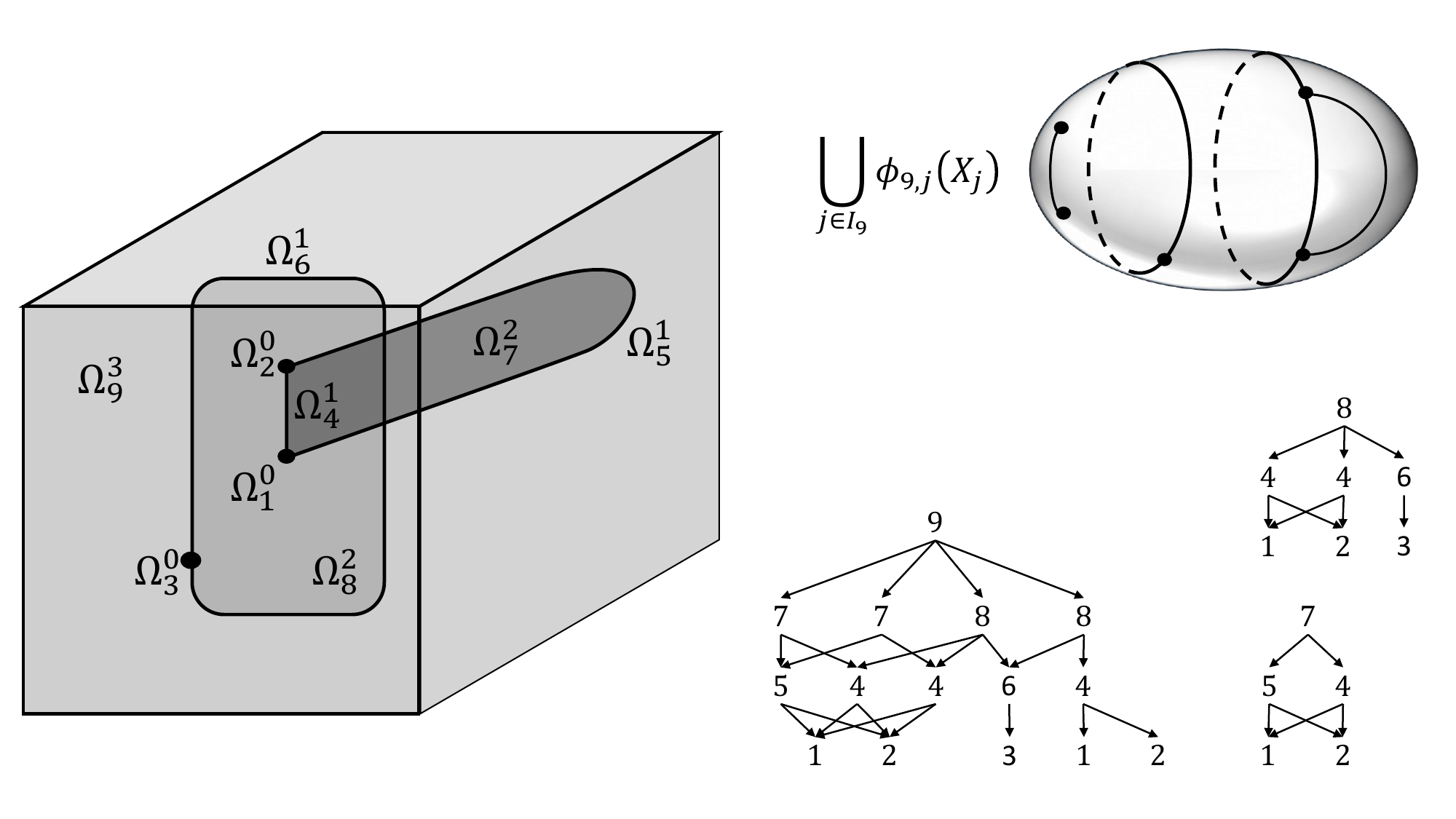}
\caption{Example of a geometry which is not a cellular complex in 3D (left). The inner boundary of $X_9^3$ (top right), and the conforming DAGs with roots in $I^2 = \{ 7, 8 \}$ and $I^3 = \{ 9 \}$ (bottom right).}
\label{fig: intersect}
\end{figure}

\begin{example}
	A self-intersecting 2D manifold in 3D is illustrated in \Cref{fig1-2}. In this example, the top-dimensional domain is not contractible and corresponds to the outside of a torus. The boundary of the top-dimensional domain (logically a torus) is shown together with its decomposition. 
\end{example}

\begin{figure}[tbhp!]
\centering
\includegraphics[width=.9\textwidth]{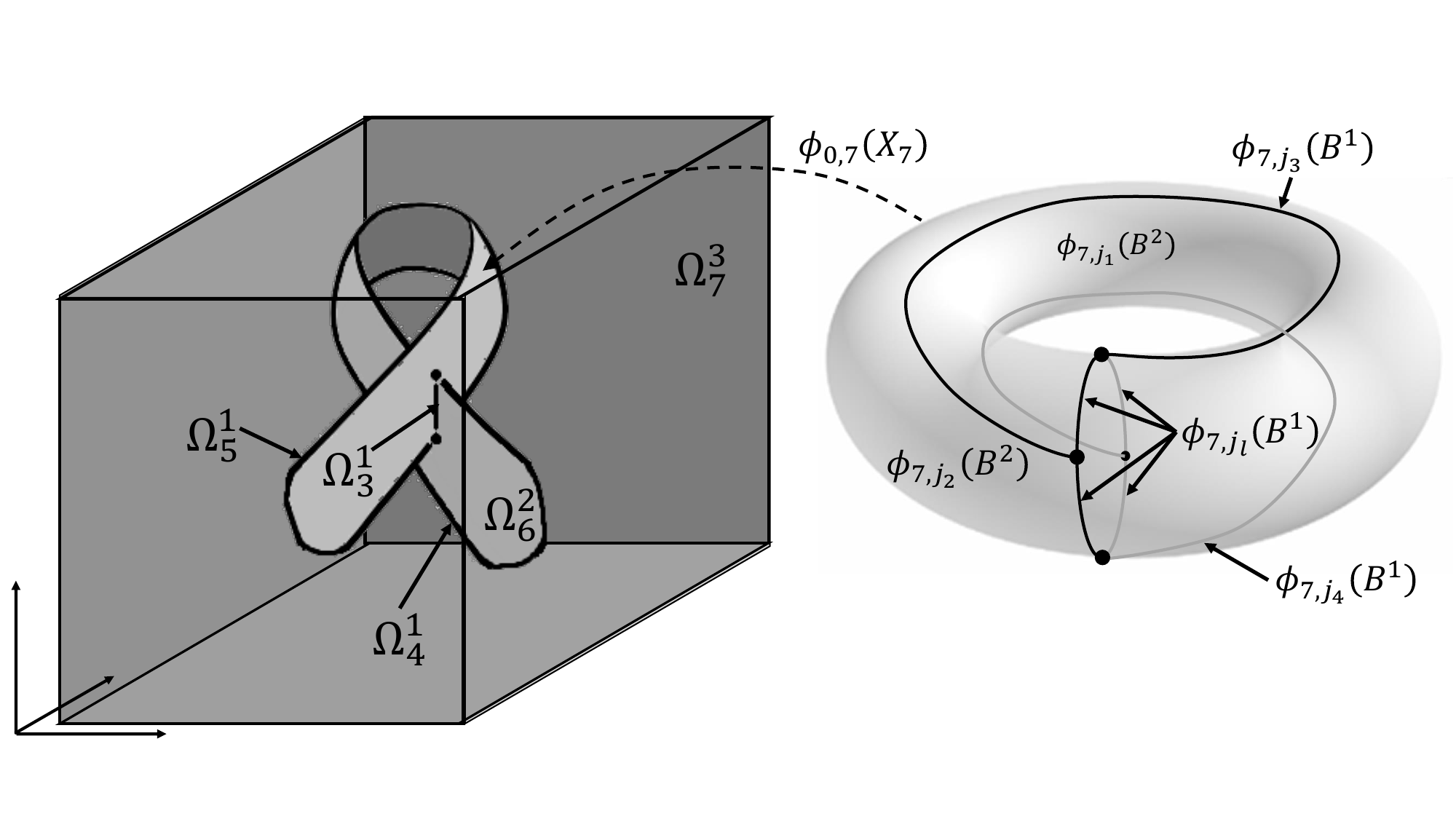}
\caption{Example of a permissible geometry in 3D (left), the partitioned inner boundary $\bigcup_{j \in I_7} \phi_{7, j}(B^{d_j}) \subset \partial X_7^3$ of the pre-image of the top-dimensional domain (right).
}\label{fig1-2}
\end{figure}

A primary object of our study will be function spaces and differential operators with respect to the domain partition $\Omega$ and the corresponding forest $\mathfrak{F}$. 

As above, we use a partial to indicate boundaries, such as
e.g.~$\partial \Omega_i$. The boundary operator on manifolds $\Omega_i$ is consistently understood in terms of the reference domains $X_i$, such that the notation $\partial_j \Omega_i$ refers to $\phi_{0, i}(\partial X_i \cap \phi_{i, j} (X_j) ) = \phi_{0, j} X_j$. Decomposition of boundaries follows the decomposition of the full domain, 
$\partial \Omega_i = \partial_i Y \bigcup_{j \in I_i}{\partial_j\Omega_i}$. 
Here $\partial_i Y = \partial_D Y \cap \partial \Omega_{s_i}$ and
$\partial_j \Omega_i$ is assumed to have the same coordinate map, up to orientation, as $\Omega_{s_j}$. We denote the
relative orientation of two such manifolds by $\varepsilon(\partial_j \Omega_i, \Omega_{s_j})$ defined as
\begin{align}
	\varepsilon(\partial_j \Omega_i, \Omega_{s_j}) 
	&= 
	\begin{cases}
		 1, & \text{if the orientations of $\partial_j \Omega_i$ and $\Omega_{s_j}$ coincide,} \\
		-1, & \text{otherwise.}
	\end{cases}
\end{align}

We make a final comment on notation. For each $j \in \mathfrak{F}$, we use a hat symbol to denote the corresponding root $\hatj \in I$ such that $j \in I_{\hatj}$.

\subsection{Fixed-dimensional exterior calculus}\label{sec2.2}

In order to fix notation, we briefly recall some results from exterior
calculus. We refer the reader to more comprehensive texts for a complete
introduction (some instructive and relevant works are \cite{6,7,8,9,10,31}, noting that in
general throughout the text we will mostly be consistent with the
notation as used in \cite{7} with the exception that for a manifold $\Omega_i$, we
denote the spaces of alternating $k$-forms
\begin{equation}
\Lambda^k(\Omega_i) 
\qquad\text{for}\qquad 
k \in \{0, 1, \ldots, d_i \}.
\label{eq2.1}
\end{equation}
We will frequently omit the dependence on the domain when no confusion
may arise. Spaces of alternating $k$-forms are connected by the wedge
product, so that for $a \in \Lambda^{k_1}$ and
$b \in \Lambda^{k_2}$, the product
$a \wedge b \in \Lambda^{k_1 + k_2}$. The wedge product has the
property that $a \wedge b = (- 1)^{k_1k_2}b \wedge a$.

We define the space of alternating $k$-forms with $m$ times differentiable coefficients as $C^m\Lambda^k$. In this work we will only consider bounded spaces of continuous functions, thus for functions in $C^m\Lambda^k$, the coefficients and their partial derivatives up to $m$-th order are taken to be finite (this space is sometimes referred to as $C^m_B\Lambda^k$, but we will omit the subscript $B$). Furthermore, under the weaker assumption that the coefficients of the alternating $k$-forms are integrable, given a basis for one-forms
$\mu_1 \ldots \mu_{d_i} \in \Lambda^1(\Omega_i)$,
we can define the weighted inner product for
$a,b \in \Lambda^k(\Omega_i)$, 

\begin{equation}
(a,b)_{\Omega_i} = \Bigl(\sum_{\sigma} {a_{\sigma} \mu_{\sigma(1)} \wedge \cdots \wedge  \mu_{\sigma(k)}},\sum_{\sigma} {b_{\sigma} \mu_{\sigma(1)} \wedge \cdots \wedge \mu_{\sigma(k)}}\Bigr)_{\Omega_i} = \int_{\Omega_i}\sum_{\sigma} {a_{\sigma} b_{\sigma} } \omega_i,
\label{eq2.2}
\end{equation}
where $\sigma$ represents all ordered selections of $k$ values from
$1\ldots d_i$, and $\omega_i \in \Lambda^{d_i}(\Omega_i)$ is
the unit volume form.

The vector spaces are pairwise dual to each other, in particular
$\Lambda^k \sim \Lambda^{d_i - k}$. Dual forms are obtained by the
Hodge star operator, such that $\mystar a \in \Lambda^k$ satisfies
\begin{equation}
	\int_{\Omega_i}{a \wedge b} = (\mystar a,b)_{\Omega_i}
	\qquad\text{for all } 
	b \in \Lambda^k. 
	\label{eq2.3}
\end{equation}
We choose the Euclidean metric in reference space $X_i$ as the underlying metric.

The inner product given in equation \eqref{eq2.2} induces a norm
 \begin{equation}
	 \| a \| = (a,a)^{1/2} 
	 \label{eq2.4}
\end{equation}
and we define the spaces of square integrable forms
\begin{equation}
	L^2\Lambda^k:\{ a \in \Lambda^k  \mid \|a\| < \infty\}.
	\label{eq2.5}
\end{equation}

For differentiable alternating forms, the exterior derivative $\myd$ maps
$\Lambda^k \rightarrow \Lambda^{k + 1}$. We define the space of
forms which have square integrable differentials as
\begin{equation}
	H \Lambda^k: \{a \in L^2\Lambda^k \mid \myd a \in L^2 \Lambda^{k + 1} \}.
\label{eq2.6}
\end{equation}
These spaces are endowed with the norm
$\| a \|_\myd = \| a \| + \| \myd a \|$.
A proper subspace of the space $H\Lambda^k$ is that which includes
natural boundary conditions with respect to the differential operator,
\begin{equation}
	\mathring{H}\Lambda^k:\{a \in H\Lambda^k \mid\Tr a = 0\}. 
	\label{eq2.7}
\end{equation}
It is important to recall that the above spaces could equivalently be defined
as the closure of $C^\infty\Lambda^k$ with respect to the stated norms \cite{rudin2006functional}. 

By definition of the exterior derivative, the following sequence is a cochain
complex, i.e.~the differential operators map
\begin{equation}
	0 \rightarrow \mathbb{R} \xrightarrow{\subset} H\Lambda^0
	\xrightarrow{\myd} H\Lambda^1 \xrightarrow{\myd}\cdots\xrightarrow{\myd} H\Lambda^d \rightarrow 0
	\label{eq2.8}
\end{equation}
and $\myd\myd a = 0$ for all $a$.

In the case of $\partial Y = \partial_D Y$, the de Rham complex is extended by
including the integral
\begin{equation}
	0 \rightarrow \mathring{H}\Lambda^0\xrightarrow{\myd}
	\mathring{H}\Lambda^1\xrightarrow{\myd}\cdots\xrightarrow{\myd}
	\mathring{H}\Lambda^d\xrightarrow{\int} \mathbb{R} \rightarrow0.
	\label{eq2.9}
\end{equation}
We refer to de Rham complexes such as \eqref{eq2.8} and \eqref{eq2.9} by the abbreviated
notation $(H\Lambda^{\bullet},\myd)$ and
$(\mathring{H}\Lambda^{\bullet},\myd)$, respectively.

For contractible domains, the function spaces on alternating $k$-forms
form an exact de Rham complex (extended in the sense of interpreting the
inclusion of constant functions as a differential operator). Thus every
closed form (i.e.~$\myd a = 0$) is exact (i.e.~$a = \myd b$ for some $b$).
This is known as the Poincar\'e Lemma. 
For general domains, the dimension
of the cohomology space will be given by the Betti numbers.

In the case where $d = n = 3$, the exterior derivative $\myd$
corresponds to $\myd \sim \{\nabla, \nabla\times, \nabla\cdot\}$
for the representatives of the $k$-forms. Furthermore, the spaces
$H\Lambda^k$ correspond to the classical spaces
$H\Lambda^0 \sim H_1$, $H\Lambda^1 \sim H(\nabla \times)$,
$H\Lambda^2 \sim H(\nabla \cdot)$, and $H\Lambda^3 \sim L^2$.
The central part of the de Rham sequence \eqref{eq2.8} takes the form:
\begin{align*}
	H_1\xrightarrow{\nabla} H(\nabla \times) \xrightarrow{\nabla \times} 
	H(\nabla \cdot)\xrightarrow{\nabla \cdot}L^2.
\end{align*}

Note that the above definitions imply, from a formal perspective, that partial integration and
Stokes' theorem are valid for $a \in H\Lambda^k$ and
$b \in H\Lambda^{d_i-k-1}$:
\begin{equation}
	\int_{\Omega_i}{\myd a \wedge b + (- 1)^ka \wedge \myd b} = 
	\int_{\Omega_i}{\myd (a \wedge b)} = 
	\int_{\partial \Omega_i}{\Tr (a \wedge b)} = 
	\int_{\partial \Omega_i}{\Tr a \wedge \Tr b}.
\label{eq2.10}
\end{equation}
In order to rigorously show that the wedge product of traces is well-defined in the final equality, a more careful treatment is needed. For the cases of primary interest in applications ($n\le3$), equation \eqref{eq2.10} can be verified explicitly. Indeed, for $k=0$ the trace spaces are duals with respect to the boundary (resp. $H^{1/2}\Lambda^0$ and $H^{-1/2}\Lambda^{n-1}$ for the trace of $H\Lambda^0$ and $H\Lambda^{n-1}$), while the case of $n=3$ and $k=1$ has been analyzed separately \cite{buffa2002traces}. We do not know of results which establish equation \eqref{eq2.10} 
rigorously for $n\geq 4$, although this result seems reasonable to conjecture. That said, our main motivation relates to physical problems, for which $n=3$, and we will therefore not elaborate further on this point. 

We will need the codifferential operator, defined as the dual of the
exterior derivative, which we denote by
$\mystar \myd^*a = (-1)^k\myd\mystar a$ for
$a \in \Lambda^k$. The codifferential induces a function space
\begin{equation}
	H^*\Lambda^k:\{ a \in L^2\Lambda^k  \mid \|\myd^*a\| < \infty\}.
\label{eq2.11}
\end{equation}
Using that
$\mystar (\mystar a) = (-1)^{k(d_i + 1)}a$,
we can now write Stokes' Theorem in terms of inner products: Let
$a \in H\Lambda^k$, and $c \in H^*\Lambda^{k + 1}$, then with
$c = \mystar b$ we calculate (with $b \in H\Lambda^{d_i-k-1})$:
\begin{equation}
	(\myd a,c)_{\Omega_i}-(a,\myd^*c)_{\Omega_i} = 
	(- 1)^{d_i k}\int_{\partial \Omega_i}{\Tr a \wedge \Tr\mystar c} = 
	(\Tr a,\Tr^*c)_{\partial \Omega_i}.
	\label{eq2.12}
\end{equation}
Here, we have introduced the dual trace (or cotrace) operator for
$e \in \Lambda^k(\Omega_i)$ such that 
\begin{equation}
	\mystar_{\partial}\Tr^*e = \Tr (\mystar e),
	\label{eq2.13}
\end{equation}
where the $\mystar_{\partial}$ is the Hodge star with respect to the
boundary. The dual trace does not appear to have standard notation, but
appears, up to sign convention, in earlier works in various forms (see e.g. \cite{11,12,30} for related
work). With our sign convention, the dual trace is commutative up to the codimension of the manifold, 
\begin{equation}
	\Tr^*_{\partial_j \Omega_i} \myd^*e = (-1)^{d_i - d_j} \myd^*\Tr^*_{\partial_j \Omega_i} e
	\label{eq2.13b}
\end{equation}

By applying \eqref{eq2.12} to $\myd\myd a$, we obtain the following
integration-by-parts formula on the boundary:
\begin{equation}
	(\Tr\myd a,\Tr^*c)_{\partial \Omega_i} + (\Tr a,\Tr^*\myd^*c)_{\partial \Omega_i} = 0.
	\label{eq2.14}
\end{equation}

For contractible domains, we have a Helmholtz decomposition, such that
for all $a \in L^2 \Lambda^k$, there exist $a_{\myd} \in H\Lambda^k$
and $a_{\myd^*} \in H^*\Lambda^k$ such that
\begin{equation}
	a = a_{\myd^*} + a_{\myd} 
	\qquad\text{while both}\qquad
	\myd^*a_{\myd^*} = 0 = \myd a_{\myd}.
	\label{eq2.15}
\end{equation}

For general domains, there may be a finite-dimensional cohomology,
with dimension given by the Betti numbers, such that we have the Hodge
decomposition
\begin{equation}
	a = a_{\myd^*} + a_{\myd} + a_0. 
	\label{eq2.16}
\end{equation}
In this case, the final term represents the cohomology class, and is a
non-trivial solution to the equations $\myd a_0 = 0 = \myd^*a_0$. For this decomposition to be unique, appropriate boundary conditions need to be imposed similar to those presented below in \Cref{thm: Hodge}.

Finally, we recall the following form of the Poincar\'e--Friedrichs
inequality: For
$a \in \mathring{H}\Lambda^k \cap H^*\Lambda^k$ or
$H\Lambda^k \cap \mathring{H}^*\Lambda^k$, it holds that
\begin{equation}
	\| a \|_{\Omega_i} \lesssim \| \myd a \|_{\Omega_i} + 
	\| \myd^*a \|_{\Omega_i} + \| a_0 \|_{\Omega_i}.
	\label{eq2.17}
\end{equation}

The results stated above represent the main tools for developing
elliptic differential equations, such as the Hodge-Laplacian (that is to
say, $\myd\myd^* + \myd^*\myd$), on manifolds. The main contribution of
this paper is to extend these results and apply them in the setting of mixed-dimensional geometries, as defined in
\Cref{sec2.1}.

\section{Differential forms}\label{sec3}

In this section we provide an extension of the exterior derivative and
the inner product to the geometry and structures of \Cref{sec2.1}, and prove properties
of the resulting operators. A main objective is to define function spaces 
on $\mathfrak{F}$ which retain the same structure as the classical function spaces of alternating
forms on $Y$. 

Semi-discrete differential operators appear in several applications. In
addition to the references cited in the introduction, similar operators
and structures to those defined in \Cref{sec3.1} have also recently been
defined in order to consider mixed-type boundary conditions in the
context of finite element exterior calculus \cite{11}. 

\subsection{Mixed-dimensional \texorpdfstring{$k$}{k}-forms}\label{sec3.1}

We are interested in differential forms over $\Omega$ which preserve
the properties known from \Cref{sec2.2}. Let us therefore first define the
integral over forests $\mathfrak{F}$, such that
\begin{equation}
\int_{\mathfrak{F}}\mathfrak{w} = \sum_{j \in \mathfrak{F}}\int_{\Omega_{s_j}} \omega_j.
\label{eq3.1}
\end{equation}
Here, we have introduced the mixed-dimensional volume form
$\mathfrak{w} = [\omega_j]_{j \in \mathfrak{F}}$. 

Recall that forms on $\Omega$ are defined with respect to $X$. I.e., a form $a \in \Lambda^k(\Omega_i)$ exists if its pull-back $\phi_{0, i}^* a \in \Lambda^k(X_i)$. More generally, since $\phi$ is a coordinate system, statements such as ``$a$ is integrable'' are always understood to mean ``$\phi^* a$ is integrable''. The same holds for statements concerning continuity and differentiability. As an immediate example, this means that
\begin{align}
	\int_{\Omega_i} a_i  &\equiv  \int_{X_i}  \phi_{0, i}^* a_i.
\end{align}

Since each
$\omega_i \in \Lambda^{d_i}(\Omega_i)$, we are motivated to define the mixed-dimensional space of $n$-forms on the forest $\mathfrak{F}$ as
\begin{equation}
\mathfrak{L}^n(\mathfrak{F}) = \prod_{i\in I}\Lambda^{d_i}(\Omega_i).
\label{eq3.2}
\end{equation}
More generally, we are interested in extending alternating $k$-forms to the mixed-dimensional
geometry. Consider therefore the following definition of alternating $k$-forms on a DAG $\mathfrak{S}_i$ for $i\in I$, 
\begin{equation}
\mathfrak{L}^k(\mathfrak{S}_i) = \Lambda^k(\Omega_i) \times \prod_{j \in I_i}
\Lambda^k(\Omega_{s_j}).
\label{eq3.3a}
\end{equation}
Here and in the following, we will use (without further comment) the convention that
$\Lambda^k(\Omega_i) = 0$ for $k \notin \{ 0, 1, \ldots, d_i \}$. Note that with this convention, we 
observe that for $k=d_i$, only the root of the DAG contributes, i.e. $\mathfrak{L}^{d_i}(\mathfrak{S}_i) = \Lambda^{d_i}(\Omega_i)$. 

By assembling over all DAGs in a forest, we then obtain mixed-dimensional alternating $k$-forms on the full forest as
\begin{equation}
\mathfrak{L}^k(\mathfrak{F}) = \prod_{i\in I}\mathfrak{L}^{k-(n-d_i)}(\mathfrak{S}_i).
\label{eq3.3}
\end{equation}
It is clear that this is a generalization of the volume forms in the sense that 
we recover equation \eqref{eq3.2} from equation \eqref{eq3.3} for $k=n$. 
From equation \eqref{eq3.3}, we note that for any root $i\in I$ and node $j\in \mathfrak{S}_i$, we denote the
degree of the associated local form in $\mathfrak{L}^k$ by $k_j = k-(n-d_i)$. It follows that $k_j$ only depends on the dimension of its root $i$, 
and not on $d_j$. 

We will consistently use Gothic letters for mixed-dimensional functions and
spaces on forests or DAGs, such as
$\mathfrak{a} \in\mathfrak{L}^k(\mathfrak{F})$, with the natural decomposition $\mathfrak{a} = [\mathfrak{a}_i]$ where  $\mathfrak{a}_i\in \mathfrak{L}^{k_i}(\mathfrak{S}_i)$. 
We use $\iota_j$ to denote the forms associated with node $j\in \mathfrak{F}$, such that $\iota_j :\mathfrak{L}^k(\mathfrak{F}) \rightarrow \Lambda^{k_j}(\Omega_{s_j})$. 
We will revert to regular Latin font for the fixed-dimensional alternating forms $a_j=\iota_j \mathfrak{a}$.

We will use three different spaces of forms on forests. These spaces generalize notions of square integrable functions, locally continuous functions, and weakly differentiable functions. First, we introduce $L^2$ functions over the mixed-dimensional structures.
\begin{definition}\label{def: L2} 
Let the space of square integrable $k$-forms over $\mathfrak{F}$ be denoted $L^2 \mathfrak{L}^k(\mathfrak{F})$ and defined as
\begin{equation*}
	L^2 \mathfrak{L}^k(\mathfrak{F}) : \{\mathfrak{a}\in\mathfrak{L}^k(\mathfrak{F}) \mid 
	a_j \in L^2 \Lambda^{k_j}(\Omega_{s_j})\quad
	\forall\ j \in \mathfrak{F} \}.
\end{equation*}
\end{definition}
The space $L^2 \mathfrak{L}^k(\mathfrak{F})$ has an inner product defined for $\mathfrak{a}, \mathfrak{b} \in L^2\mathfrak{L}^k(\mathfrak{F})$
\begin{equation}
	(\mathfrak{a},\mathfrak{b})_{\mathfrak{F}} = 
	\sum_{j\in \mathfrak{F}} (a_j,b_j)_{\Omega_{s_j}} 
	\label{eq:inner product}
\end{equation}
and the inner product induces a norm on $L^2\mathfrak{L}^k(\mathfrak{F})$
\begin{equation}\label{eq:L2norm}
	\|\mathfrak{a}\| = (\mathfrak{a},\mathfrak{a})_{\mathfrak{F}}^{1/2}.
\end{equation}

\begin{remark}
	The definition of the inner product naturally depends on the underlying metric. Since all $X_i$ are embedded in $\mathbb{R}^{d_i}$, we use the Euclidean metric in reference space.
\end{remark}

As in the fixed-dimensional case, in order to treat degenerate coefficients as appear
in different physical regimes, weighted inner products may be desirable in
applications (for a more detailed discussion, see \Cref{sec4} and also \cite{13}).

\begin{lemma}\label{lem: inner product} 
	The inner product defined in equation \eqref{eq:inner product} is
	symmetric, linear, and positive-definite.
\end{lemma}

\begin{lemma}\label{lem: L2L Hilbert}
	The mixed-dimensional space $L^2\mathfrak{L}^k(\mathfrak{F})$ is a Hilbert space.
\end{lemma}

\begin{proof}
	Due to the inner product \eqref{eq:inner product}, $L^2\mathfrak{L}^k(\mathfrak{F})$ is pre-Hilbert.
	Completeness follows by the product structure given by \Cref{def: L2}.
\qed
\end{proof}

\subsection{Strongly differentiable \texorpdfstring{$k$}{k}-forms}\label{sec3.2}

We first define a notion of locally continuous forms. 
\begin{definition}\label{def:cont} 
	Let the space of locally continuous mixed-dimensional $k$-forms over $\mathfrak{F}$ be denoted $C \mathfrak{L}^k(\mathfrak{F})$, defined such that 
	\begin{align}
		C \mathfrak{L}^k(\mathfrak{F}) : \{\mathfrak{a} \in \mathfrak{L}^k(\mathfrak{F}) \mid &
		a_j \in C^\infty \Lambda^{k_j}(\Omega_{s_j}) \text{ and } \nonumber\\
		& 
		a_j = \varepsilon(\Omega_{s_j},\partial_{j} \Omega_i) \Tr_{\partial_j \Omega_i} a_i  \quad
		\forall\ i\in I\ \text{ and }\ j\in \mathfrak{S}_i 
		 \}.
	\end{align}
	We denote by $\mathring{C} \mathfrak{L}^k$ the subset of forms $\mathfrak{a}\in C \mathfrak{L}^k$ such that $\Tr a_i=0$ on $\partial_i Y$. 
\end{definition}

\begin{lemma}
	On each DAG $\mathfrak{S}_i$ with $i\in I$, the space $C \mathfrak{L}^k(\mathfrak{S}_i)$ is isomorphic to
	$C^\infty \Lambda^k_i\big(\overline{X}_i \setminus  \phi_{0,i}^{-1}(\partial_i Y)\big)$.
\end{lemma}
\begin{proof}
	Follows from \Cref{def:DAG} and the defined maps $\phi_{0,i}$ \cite{31}.
\qed
\end{proof}

Locally continuous forms are interpreted as (bounded) continuous $k_i$-forms on each $\Omega_i$, for $i\in I$, with continuous extensions onto the boundaries as appropriate. We provide two examples, as illustrated for $n = 1$ in \Cref{fig:Ex1}. 

\begin{figure}[tbhp!]
\centering
\includegraphics[width=.5\textwidth]{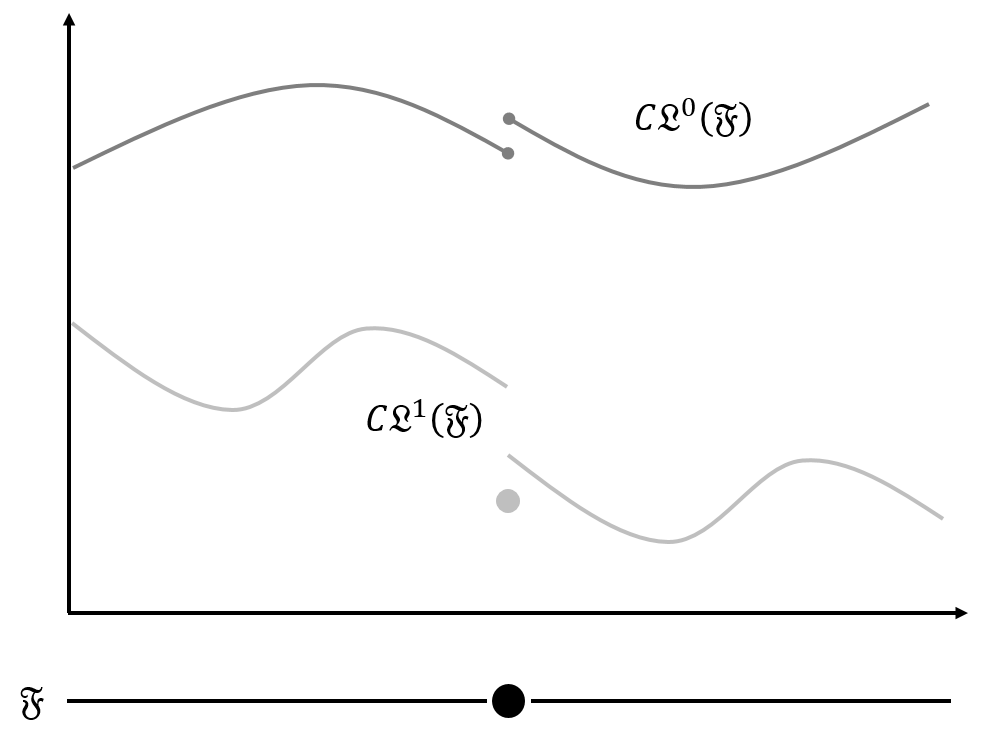}
\caption{Example of locally continuous functions in 1D. In this illustration, we have below the diagram indicated a domain decomposed into two 1D segments separated by a 0D domain. The functions belonging to $C \mathfrak{L}^0$ and $C \mathfrak{L}^n$ are further described in \Cref{ex: CL0} and \Cref{ex: CLn}, respectively.
}\label{fig:Ex1}
\end{figure}

\begin{example} \label{ex: CL0}
	For $\mathfrak{a} \in C \mathfrak{L}^0(\mathfrak{F})$ and $i\in I^n$, we have that $a_i$ is a continuous, infinitely differentiable, function on $\Omega_i$. Discontinuities are permitted across all $\Omega_j$ for $j\in I^d$ where $d<n$, and the forms on $\Omega_j$ are void. In contrast, for all $l\in I_i$ the forms on $\partial_l \Omega_i$ are continuous and appropriately interpreted as traces of $a_i$ (up to sign). 
\end{example}

\begin{example} \label{ex: CLn}
	The volume forms $\mathfrak{a}\in C \mathfrak{L}^n(\mathfrak{F})$ are piecewise continuous functions, in the sense that $a_i$ are infinitely differentiable 
	on each $\Omega_i$, for $i\in I$. Moreover, all $a_i$ become volume forms of the appropriate dimension. Thus, for all $l \in I_i$, the forms $a_l$ are void. 
\end{example}

\begin{theorem}\label{thm: C-dense-L2}
	The space $C \mathfrak{L}^k$ is dense in $L^2\mathfrak{L}^k$ with respect to the norm \eqref{eq:L2norm}. 
\end{theorem}

\begin{proof}
	We provide a constructive proof exploiting the product structure of $L^2\mathfrak{L}^k$. Let $\mathfrak{a}\in L^2\mathfrak{L}^k$, and we will construct $\mathfrak{b}\in C \mathfrak{L}^k$ such that $\|\mathfrak{a-b}\|<\epsilon$, for any $\epsilon>0$. Throughout the proof, we will work on the reference spaces $X_i$ in the sense of \Cref{sec2.1}. Furthermore, let $\epsilon' = \epsilon/|\mathfrak{F}|$, where $|\mathfrak{F}|$ is the number of nodes in $\mathfrak{F}$. 

	Consider first manifolds $j\in \mathfrak{F}$ such that $d_j = k_j$. Then for all $j'\in I_j$, $d_{j'}<d_j=k_j=k_{j'}$, and thus the local forms $a_{j'}$ for $j'$ are all void, while the local forms $a_j$ are in the usual spaces $L^2 \Lambda^{k_j}(X_j)$. Since $C^\infty \Lambda^{k_j}(X_j)$ is dense in $L^2 \Lambda^{k_j}(X_j)$, we can choose $b_j\in C^\infty \Lambda^{k_j}(X_j)$ such that $\|a_j-b_j\|_{X_j}<\epsilon'$. 

	Proceeding recursively, let $j\in \mathfrak{F}$ such that $b_{j'}$ is defined for all $j'\in I_j$. Thus, since $\partial X = \bigcup_{j'\in I_j} \phi_{j,j'}(X_{j'})$, we can construct smooth extensions of the boundary data $b_{j'}$ into $X_j$ \cite{12}. Let $c_j$ be any such smooth extension, such that $\varepsilon(\Omega_{s_{j'}},\partial_{j'} \Omega_{s_j}) \Tr_{\partial_{j'} X_j} c_j = \phi_{j,j'} b_{j'}$. Then we construct $a'_j = a_j - c_j \in L^2 \Lambda^{k_j}(X_j)$, and due to the density of $C_0^\infty$ in $L^2$, we can choose $b'_j\in C_0^\infty \Lambda^{k_j}(X_j)$ such that  $\|a'_j-b'_j\|_{X_j}<\epsilon'$. We then define $b_j = b'_j+c_j \in C^\infty \Lambda^{k_j}(X_j)$, and it follows trivially that $\|a_j-b_j\|_{X_j}<\epsilon'$. 

	Finally, let $\mathfrak{b}\in \mathfrak{L}^k$ be defined such that $\iota_j \mathfrak{b} = b_j$. By construction, $\|\mathfrak{a-b}\| < \epsilon$, and it remains to verify that $\mathfrak{b}\in C \mathfrak{L}^k$. With reference to \Cref{def:cont}, this holds, since all $\iota_j \mathfrak{b}\in C^\infty \Lambda^{k_j}(X_j)$, and since $\iota_j \mathfrak{b} = \varepsilon(\Omega_{s_j},\partial_{j} \Omega_i) \Tr_{\partial_{j} \Omega_i} \iota_i \mathfrak{b}$ by the construction above.
\qed
\end{proof}

For continuous mixed-dimensional $k$-forms, we can define an appropriate exterior derivative in the same sense as \cite{11,5}. 
\begin{definition}\label{def:strongdifferential} 
	For $\mathfrak{a}\in C \mathfrak{L}^k(\mathfrak{F})$, let the strong form of the mixed-dimensional exterior derivative $\mathfrak{d}\colon C \mathfrak{L}^k(\mathfrak{F}) \rightarrow C \mathfrak{L}^{k+1}(\mathfrak{F})$ be defined for all $j \in \mathfrak{F}$ as
	\begin{equation*}
		\iota_j (\mathfrak{d}\mathfrak{a}) = \myd a_j + \iota_j (\Bbbd \mathfrak{a}),
	\end{equation*}
	where the discrete differential operator 
	$\Bbbd\colon C \mathfrak{L}^k(\mathfrak{F}) \rightarrow C \mathfrak{L}^{k+1}(\mathfrak{F})$
	is defined on the roots $i \in I$ by
	\begin{equation}
		\iota_i (\Bbbd \mathfrak{a}) = 
		(-1)^{n - k} \sum_{\{l \in \mathfrak{F}| s_l = i \text{ and } \hat{l} \in I^{d_i + 1} \}} \varepsilon(\Omega_i, \partial_l \Omega_{\hat{l}}) a_l
		\label{eq3.8}
	 \end{equation}
	and subsequently on all branches $j \in I_i$ as
	\begin{align}
		\iota_j (\Bbbd \mathfrak{a}) = \varepsilon(\Omega_{s_j},\partial_{j} \Omega_i) \Tr_{\partial_{j} \Omega_i} \iota_i (\Bbbd \mathfrak{a}) \label{eq3.8b}
	\end{align}
\end{definition}
\begin{remark}
This definition is equivalent to constructing a single graded complex from the anti-diagonals of the double complex induced by $\myd$ and $\Bbbd$, as elaborated in \cite{6}. In order to be self-contained, we provide the explicit construction here.
\end{remark}

Due to the continuity properties of $C \mathfrak{L}^k$, the discrete operator $\Bbbd$ can be expressed locally for each node $j$ in the forest $\mathfrak{F}$. The local summation is then performed over an index set $\gamma_j^{-1}$ which we define below.

\begin{corollary}
	For each $j \in \mathfrak{F}$, we can define a minimal set of indices, denoted $\gamma_j^{-1}$, such that for all $\mathfrak{a} \in C \mathfrak{L}^k$, it holds that
	\begin{align*}
		\iota_j (\Bbbd \mathfrak{a}) &= (-1)^{n - k} \sum_{l \in \gamma_j^{-1}} \varepsilon_{jl} a_l
	\end{align*}
	with $\varepsilon_{jl} \in \{ \pm 1 \}$. More precisely, it follows from \eqref{eq3.8} and \eqref{eq3.8b} that $\varepsilon_{jl} = \varepsilon(\Omega_{\hatj} , \partial_{l'} \Omega_{\hat{l}})$ where the hat denotes the corresponding root and $l'$ is the appropriate index in the set defined in \eqref{eq3.8}.
\end{corollary}

\begin{example}
	For all roots $i \in I$, we have $\gamma_i^{-1} = \{l \in \mathfrak{F}| s_l = i \text{ and } \hat{l} \in I^{d_i + 1} \}$ as in \eqref{eq3.8}, i.e., the set consists of all branches which geometrically coincide with $\Omega_i$ and have a root of dimension $d_i + 1$. In the example from \Cref{fig: slit}, $\gamma_3^{-1}$ therefore consists of the indices $j_3, j_4 \in I_4$ which have $s_{j_3} = s_{j_4} = 3$.
\end{example}


\begin{example}
	The set $\gamma_j^{-1}$ is void for branches $j \in I_i$ for which $\iota_j (\Bbbd \mathfrak{a}) = 0$ for all $\mathfrak{a} \in C \mathfrak{L}^k$. An example of this arises in \Cref{fig: slit} where all $\mathfrak{a} \in C \mathfrak{L}^0$ map to zero at the extremities of $\Omega_3$, i.e. on the branches $j \in I_3$ with $d_j = 0$.
\end{example}

The strong form of the mixed-dimensional exterior derivative is thus interpreted as the fixed-dimensional exterior derivative within each domain $X_i$, where the out-of-manifold components of the differential are expressed in terms of the traces of values on the manifolds which are in the neighborhood of $\Omega_i$. This definition is consistent with standard models for materials with thin inclusions \cite{martin2005modeling,3,13}. Note that it is clear that the differential operator preserves continuity.

In absence of Dirichlet boundaries, i.e. $Y_D= \emptyset$, the mixed-dimensional spaces $C \mathfrak{L}^k$ as well as their relations given through the mixed-dimensional exterior derivative given in \Cref{def:strongdifferential}, are summarized in the
following diagram
\begin{equation}
	\begin{tikzcd}
		C \mathfrak{L}^0 \arrow[d,"\mathfrak{d}"] 
		& \tx{C^\infty\Lambda^0 (\overline{X}^n)} \arrow[d,"\myd"]\arrow[rd,"\Bbbd"]\\
		C \mathfrak{L}^1 \arrow[d,"\mathfrak{d}"] 
		& \tx{C^\infty\Lambda^1 (\overline{X}^n)} \arrow[d,"\myd"]\arrow[rd,"\Bbbd"] 
		& \tx{C^\infty\Lambda^0 (\overline{X}^{n-1})} \arrow[d,"\myd"]\arrow[rd,"\Bbbd"]\\
		C \mathfrak{L}^2 \arrow[d,dashrightarrow,"\mathfrak{d}"] 
		& \tx{C^\infty\Lambda^2 (\overline{X}^n)} \arrow[d,dashrightarrow,"\myd"]\arrow[rd, dashrightarrow,"\Bbbd"] 
		& \tx{C^\infty\Lambda^1 (\overline{X}^{n-1})} \arrow[d,dashrightarrow,"\myd"]\arrow[rd,dashrightarrow,"\Bbbd"] 
		& \tx{C^\infty\Lambda^0 (\overline{X}^{n-2})} \arrow[d,dashrightarrow,"\myd"]\arrow[rd,dashrightarrow,"\Bbbd"]\\
		C \mathfrak{L}^n 
		& \tx{C^\infty\Lambda^n (\overline{X}^n)}
		& \tx{C^\infty\Lambda^{n-1} (\overline{X}^{n-1})} 
		& \tx{C^\infty\Lambda^{n-2} (\overline{X}^{n-2})}
		& \tx{C^\infty\Lambda^0 (\overline{X}^0)}
	\end{tikzcd}
	\label{eq3.14}
\end{equation}
In order to provide this diagram, we have used the notation $X^d = \bigcup_{i\in I^d} X_i$. This diagram can be seen together with the forest in \Cref{fig: slit}.

\begin{lemma}
	The mixed-dimensional exterior derivative gives $C \mathfrak{L}^k(\mathfrak{F})$ the structure of a cochain complex, i.e. $\mathfrak{dda} = 0$ for all $\mathfrak{a} \in C \mathfrak{L}^k(\mathfrak{F})$.
\end{lemma}
\begin{proof}
	By an explicit calculation, we have for arbitrary $j \in \mathfrak{F}$:
	\begin{align}
		\iota_j (\mathfrak{d}^2\mathfrak{a}) 
		= \myd^2 a_j + \iota_j (\myd \Bbbd \mathfrak{a}) + \iota_j (\Bbbd \myd \mathfrak{a}) +  \iota_j (\Bbbd^2 \mathfrak{a}).
	\end{align}
	The first term is zero by the properties of the fixed-dimensional exterior derivative $\myd$. Additionally, the exterior derivative and the jump operator $\Bbbd$ are anticommutative:
	\begin{equation} \label{eq: anticommutative}
		\iota_j (\myd \Bbbd \mathfrak{a}) 
		= \myd \left( (-1)^{n - k} \sum_{l \in \gamma_j^{-1}} \varepsilon_{jl} a_l \right)
		= - (-1)^{n - (k + 1)} \sum_{l \in \gamma_j^{-1}} \varepsilon_{jl} \iota_l (\myd \mathfrak{a})
		= - \iota_j (\Bbbd \myd \mathfrak{a}),
	\end{equation}
	hence the second and third terms cancel. 
	
	Finally, the last term becomes
	\begin{align}
		\iota_j (\Bbbd (\Bbbd \mathfrak{a}))
		&= 
		(-1)^{n - k} \sum_{l \in \gamma_j^{-1}} \varepsilon_{jl} \iota_{l} (\Bbbd \mathfrak{a})
		= - \sum_{l \in \gamma_j^{-1}}
		\sum_{l' \in \gamma_{l}^{-1}} 
		\varepsilon_{jl} \varepsilon_{ll'} a_{l'}
		= 0.  \label{eqA.3}
	\end{align}
		The last equality holds, since from the geometry, we see that each $a_{l'}$ appears twice. The signs must be opposite depending on which
		intermediate manifold is used when taking boundary traces from $\Omega_{\hat{l'}}$, thus for $l_1, l_2 \in \gamma_j^{-1}$ with $l_1 \ne l_2$:
	\begin{equation}
		\varepsilon_{j l_1} \varepsilon_{l_1 l'} =
		- \varepsilon_{j l_2} \varepsilon_{l_2 l'}.
	\label{eqA.4}
	\end{equation}
\qed
\end{proof}

From the definitions and lemma above, the inclusion and mixed-dimensional exterior derivative $\mathfrak{d}$ lead to a de Rham complex
$(C \mathfrak{L}^{\bullet},\mathfrak{d})$:
\begin{equation}
	0 \rightarrow \mathbb{R}\xrightarrow{\subset}
	C \mathfrak{L}^0\xrightarrow{\mathfrak{d}}
	C \mathfrak{L}^1\xrightarrow{\mathfrak{d}}\cdots\xrightarrow{\mathfrak{d}} 
	C \mathfrak{L}^n \rightarrow 0.
	\label{eq: De Rham C}
\end{equation}

It will be of interest to have an explicit representation of a codifferential operator, which is consistent with an integration-by-parts formula with respect to the inner product \eqref{eq:inner product}. We therefore propose the following
\begin{definition}\label{def:strongcodifferential} 
	For $\mathfrak{a}\in C \mathfrak{L}^k(\mathfrak{F})$, let the strong form of the mixed-dimensional exterior coderivative $\mathfrak{d}^*\colon C \mathfrak{L}^k(\mathfrak{F}) \rightarrow L^2\mathfrak{L}^{k-1}(\mathfrak{F})$ be defined such that for all $\mathfrak{b}\in \mathring{C} \mathfrak{L}^{k-1}(\mathfrak{F})$
	\begin{equation} \label{int by parts}
		(\mathfrak{d}^* \mathfrak{a},  \mathfrak{b})_\mathfrak{F} = ( \mathfrak{a},  \mathfrak{d}\mathfrak{b})_\mathfrak{F}
	\end{equation}
\end{definition}

\begin{lemma}\label{lem: IBPcont}
	For $\mathfrak{a}\in C \mathfrak{L}^k(\mathfrak{F})$, the strong form of the mixed-dimensional exterior coderivative has the explicit representation
	for all $j \in \mathfrak{F}$ as 
	\begin{align}
	\iota_j (\mathfrak{d}^*\mathfrak{a}) = \myd^* a_j + \sum_{l\in J_j^{d_j+1}} \Tr_j^* a_l 
	+ (-1)^{n - k} \sum_{l \in \gamma_j}
	\varepsilon_{lj} a_l
	\end{align}
	with $\gamma_j = \left\{ l \in \mathfrak{F} | j \in \gamma_l^{-1} \right\}$ and $J_j^{d_j + 1} = \left\{ l \in \mathfrak{F} | j \in I_l^{d_l - 1} \right\}$.
\end{lemma}
\begin{proof}
	By definition of the mixed-dimensional exterior derivative and inner product, we calculate for $\mathfrak{b} \in \mathring{C}\mathfrak{L}^{k - 1}$
	\begin{align}
		(\mathfrak{db},  \mathfrak{a})_\mathfrak{F} &=
		\sum_{j \in \mathfrak{F}} (\myd b_j + \iota_j (\Bbbd \mathfrak{b}), a_j)_{\Omega_{s_j}} 
		\nonumber \\
		&=
		\sum_{j \in \mathfrak{F}} (\myd b_j + (-1)^{n - k} \sum_{l\in\gamma_j^{-1}} \varepsilon_{jl}
		b_l, a_j)_{\Omega_{s_j}} 
		\nonumber \\
		&=
		\sum_{j \in \mathfrak{F}} (b_j,\myd^* a_j)_{\Omega_{s_j}}  
		+ \sum_{l\in I_j^{d_j-1}}(\Tr_l  b_j,\Tr_l^* a_j)_{\partial_l\Omega_{s_j}} 
		+ (-1)^{n - k} \sum_{l\in\gamma_j}
		(b_j, \varepsilon_{lj} a_l)_{\Omega_{s_j}} 
		\nonumber \\
		&=
		\sum_{j \in \mathfrak{F}} (b_j,\myd^* a_j)_{\Omega_{s_j}}  
		+ \sum_{l\in J_j^{d_j+1}}(b_j,\Tr_j^* a_l)_{\Omega_{s_j}} 
		+ (-1)^{n - k} \sum_{l\in\gamma_j} 
		(b_j, \varepsilon_{lj} a_l)_{\Omega_{s_j}} 
	\end{align}
	It follows that the strong form proposed in \Cref{lem: IBPcont} satisfies \Cref{def:strongcodifferential}. It remains to show uniqueness. Let $\mathfrak{c}\in L^2\mathfrak{L}^{k-1}(\mathfrak{F})$ be any other codifferential of $\mathfrak{a}$. Then by \Cref{def:strongcodifferential} 
	$$
	(\mathfrak{c}-\mathfrak{d}^* \mathfrak{a},  \mathfrak{b})_\mathfrak{F} = 0
	$$
	for all $\mathfrak{b} \in \mathring{C} \mathfrak{L}^{k-1}(\mathfrak{F})$. By the definition of the inner product, it then follows that $c_i - \iota_i (\mathfrak{d}^* \mathfrak{a})$ is zero almost everywhere on $\Omega_{s_i}$ for all $i \in \mathfrak{F}$, thus $\mathfrak{d}^* \mathfrak{a}$ is unique up to equivalence classes in $L^2\mathfrak{L}^{k-1}(\mathfrak{F})$. 
\qed
\end{proof}

\begin{remark}
	$\mathfrak{d}^* C \mathfrak{L}^k$ has higher regularity than $L^2 \mathfrak{L}^{k - 1}$, but we will not discuss this space further.
\end{remark}

The codifferential operator suggests the following definition of spaces of codifferentiable mixed-dimensional forms. 
\begin{definition}\label{def:cont2} 
	Let the space of codifferentiable mixed-dimensional $k$-forms over $\mathfrak{F}$ be denoted $C^* \mathfrak{L}^k(\mathfrak{F})$, defined such that 
	\begin{equation}
		C^* \mathfrak{L}^k(\mathfrak{F}) : \{\mathfrak{a} \in C \mathfrak{L}^k(\mathfrak{F}) \mid 
		\mathfrak{d^*a} \in C \mathfrak{L}^{k-1}(\mathfrak{F})
		 \}.
	\end{equation}
	We denote by $\mathring{C}^* \mathfrak{L}^k$ the subset of functions $\mathfrak{a}\in C^* \mathfrak{L}^k$ such that $\Tr^* a_i=0$ on $\partial_i Y$. 
\end{definition}

\begin{lemma}\label{lem: cont_codd0}
	For $\mathfrak{a} \in C^* \mathfrak{L}^k(\mathfrak{F})$, it holds that $\mathfrak{d^*d^*a}=0$, and thus $\mathfrak{d^*} \colon C^*\mathfrak{L}^k(\mathfrak{F}) \rightarrow C^*\mathfrak{L}^{k-1}(\mathfrak{F})$. 
\end{lemma}
\begin{proof}
The statement that $\mathfrak{d^*d^*a}=0$ follows by applying \Cref{def:strongcodifferential} twice, and using that $\mathfrak{dd}=0$. The second statement follows since $0\in C \mathfrak{L}^{k-2}(\mathfrak{F})$, thus $\mathfrak{d^* a} \in C^* \mathfrak{L}^{k - 1}(\mathfrak{F})$.
\qed
\end{proof}

The codifferential operator and spaces also form a de Rham complex
$(C^*\mathfrak{L}^{\bullet},\mathfrak{d}^*)$:
\begin{equation}
	0 \leftarrow 
	C^*\mathfrak{L}^0\xleftarrow{\mathfrak{d}^*}
	C^*\mathfrak{L}^1\xleftarrow{\mathfrak{d}^*} 
	\cdots \xleftarrow{\mathfrak{d}^*} 
	C^*\mathfrak{L}^n \xleftarrow{\supset}\mathbb{R}
	\leftarrow 0.
	\label{eq: De Rham C*}
\end{equation}
The de Rham complexes with boundary conditions
$(\mathring{C}\mathfrak{L}^{\bullet},\mathfrak{d})$ and $(\mathring{C}^*\mathfrak{L}^{\bullet},\mathfrak{d}^*)$ 
are defined similarly, but with the integral instead of the inclusions, equivalent to \eqref{eq2.8} and \eqref{eq2.9}.

\begin{remark}\label{cont_characterization}
	The forms $\mathfrak{a} \in C^* \mathfrak{L}^k(\mathfrak{F})$ satisfy, in addition to the conditions of \Cref{def:cont}, also the analogous conditions related to the codifferential, i.e. that for all $i\in I$ and $j\in I_i$
	\begin{equation*}
		\iota_j (\mathfrak{d}^*\mathfrak{a}) = \varepsilon(\Omega_{s_j},\partial_{j} \Omega_i) \Tr_{\partial_{j} \Omega_i}
		\iota_i (\mathfrak{d}^*\mathfrak{a}). 
	\end{equation*} 
	Using \Cref{lem: IBPcont}, this provides an explicit constraint on the form $\mathfrak{a}$ in the sense that 
	\begin{equation*}
		\Tr_{\partial_{j} \Omega_i} \myd^* a_i = 
		\varepsilon(\Omega_{s_j},\partial_{j} \Omega_i) \left[
		\myd^* a_j + \sum_{l\in J_j^{d_j+1}} \Tr_j^*a_l + (-1)^{n - k} \sum_{l\in\gamma_j} \varepsilon_{lj} a_l\right].
	\end{equation*}
\end{remark}

\subsection{Weak differentials and Hilbert spaces}

In order to provide a suitable framework for working with partial differential equations, we are also interested in weak forms of the function spaces and operators introduced above. 

\begin{definition}\label{def:weakdifferential} 
	For $\mathfrak{a}\in L^2\mathfrak{L}^k(\mathfrak{F})$, with $a_j\in H\Lambda^{k_j}(\Omega_{s_j})$ for all $j \in \mathfrak{F}$, let a weak mixed-dimensional exterior derivative of $\mathfrak{a}$, if it exists, be any form $\mathfrak{da}\in L^2\mathfrak{L}^{k+1}(\mathfrak{F})$ such that for all $\mathfrak{b}\in \mathring{C}\mathfrak{L}^{k+1}(\mathfrak{F})$,
	\begin{equation}
		(\mathfrak{da},\mathfrak{b})_{\mathfrak{F}} = 
		(\mathfrak{a}, \mathfrak{d}^* \mathfrak{b})_{\mathfrak{F}}.
	\end{equation}
\end{definition}

\begin{definition}\label{def: HL} 
We denote the space of weakly differentiable mixed-dimensional forms on the forest $\mathfrak{F}$ as
\begin{equation}
	H\mathfrak{L}^k(\mathfrak{F}):\{\mathfrak{a} \in L^2\mathfrak{L}^k(\mathfrak{F}) \mid 
	\mathfrak{d}\mathfrak{a} \in L^2\mathfrak{L}^{k + 1}(\mathfrak{F}) \}.
	\label{eq3.11}
\end{equation}
We allow for boundary conditions on the external boundary $\partial_D Y$ in the sense of
\begin{equation}
	\mathring{H}\mathfrak{L}^k(\mathfrak{F}):\{\mathfrak{a }\in H\mathfrak{L}^k(\mathfrak{F}) \mid 
	\Tr_{\partial_i Y} a_i = 0  
	\text{ for all } i \in \mathfrak{F} \}.
	\label{eq3.12}
\end{equation}
\end{definition}
For both $H\mathfrak{L}^k$ and
$\mathring{H}\mathfrak{L}^k$ the natural norm is given as
\begin{equation}
\|\mathfrak{a}\|_{\mathfrak{d}} = \|\mathfrak{a}\| + \|\mathfrak{d}\mathfrak{a}\|.
\label{eq3.13}
\end{equation}

\begin{lemma}\label{lem: HL Hilbert} 
The mixed-dimensional spaces $H\mathfrak{L}^k(\mathfrak{F})$ and $\mathring{H}\mathfrak{L}^k(\mathfrak{F})$ are Hilbert spaces.
\end{lemma}
\begin{proof}
The case $k=n$ is immediate from \Cref{lem: L2L Hilbert}, as $H\mathfrak{L}^n(\mathfrak{F})=L^2\mathfrak{L}^n(\mathfrak{F})$. For the general case, consider any Cauchy sequence $\mathfrak{a}_l\in H\mathfrak{L}^k(\mathfrak{F})$. Then due to the completeness of $L^2\mathfrak{L}^k(\mathfrak{F})$, the limits $\overline{\mathfrak{a}}= \lim_{l\rightarrow\infty}\mathfrak{a}_l$ and $\overline{\mathfrak{da}}= \lim_{l\rightarrow\infty}\mathfrak{da}_l$ exist. It remains to show that $\overline{\mathfrak{da}} = \mathfrak{d}\overline{\mathfrak{a}}$. This holds by a standard calculation, since for any $\mathfrak{c}\in \mathring{C} \mathfrak{L}^{k+1}(\mathfrak{F})$ 
\begin{equation}
	\begin{aligned}
	(\overline{\mathfrak{da}}, \mathfrak{c})_{\mathfrak{F}} &=
	\lim_{l\rightarrow\infty} (\mathfrak{da}_l, \mathfrak{c})_{\mathfrak{F}}  
	=
	\lim_{l\rightarrow\infty} (\mathfrak{a}_l, \mathfrak{d^*c})_{\mathfrak{F}}  
	=
	(\overline{\mathfrak{a}}, \mathfrak{d^*c})_{\mathfrak{F}}  
	=
	(\mathfrak{d}\overline{\mathfrak{a}}, \mathfrak{c})_{\mathfrak{F}}
	\end{aligned}
\end{equation}
The same calculation holds for $\mathring{H}\mathfrak{L}^k(\mathfrak{F})$ when $C \mathfrak{L}^{k+1}(\mathfrak{F})$ is used as the space of test functions. 
\qed
\end{proof}

In order to provide a characterization of $H\mathfrak{L}^k(\mathfrak{F})$ in terms of local spaces on each $\Omega_i$, we introduce a 
subspace of the standard space $H\Lambda^k$. 
Thus, for each root $i\in I$, we consider functions $a_i$ from the space of weakly differentiable $k_i$-forms subject to boundary conditions imposed by the forms $a_j$ for all $j\in I_i$. Locally, we refer to these spaces with imposed trace regularity using the recursive definition
\begin{equation}
	H \Lambda^k (\Omega_i, \Tr) = 
	\{a \in H \Lambda^k(\Omega_i)| \Tr_{\partial_j \Omega_i} a \in H \Lambda^k(\Omega_{s_j}, \Tr) \  \forall j \in I_i\}
\end{equation}
with a naturally induced, recursively defined, norm
\begin{equation} \label{eq: norm HL Tr}
	\| a \|_{H \Lambda^k (\Omega_i, \Tr)} = \| a \|_{H \Lambda^k(\Omega_i)} + \sum_{j \in I_i} \| \Tr_{\partial_j \Omega_{s_i}} a \|_{H \Lambda^k (\Omega_{s_j}, \Tr)}.
\end{equation}

\begin{example}
For $k = d_i$, the space $H \Lambda^k (\Omega_i, \Tr)$ corresponds to the classical space $L^2(\Omega_i)$. 
\end{example}

\begin{example}
For $d_i \geq 2$ and $k = d_i - 1$, the space $H \Lambda^k (\Omega_i, \Tr)$ corresponds to the subspace of the classical space $H(\nabla\cdot;\ \Omega_i)$ with the restriction that its traces are square integrable. It can thus be described by
\begin{equation}
	H \Lambda^{d_i - 1} (\Omega_i, \Tr) = 
	\{a_i \in H(\nabla\cdot;\ \Omega_i)| \Tr_{\partial_j \Omega_i} a_i \in L^2(\Omega_{s_j}) \ \forall j \in I_i^{d_i - 1}\}.
\end{equation}
This is in contrast to the full space $H(\nabla\cdot;\ \Omega_i)$, which contains function traces with only $H^{-1/2}$ regularity. Although there is no common notation for spaces with enhanced boundary regularity, they appear in applications, see e.g. \cite{13,martin2005modeling,24}.
\end{example}

\begin{lemma} \label{lem: extensions}
	There exist bounded extension operators $\mathcal{R} : H \Lambda^k(\partial X_i) \rightarrow H\Lambda^k(X_i)$.
\end{lemma}
\begin{proof}
	See e.g. \cite{12}.
\qed
\end{proof}

\begin{lemma}\label{tr-H}
$H \Lambda^k (\Omega_i, \Tr)$ is a Hilbert space. 
\end{lemma}
\begin{proof}
	Consider $i\in I$, $j \in I_i$, and let first $k=d_i-1$. Then $\phi_{i,j}(X_j)$ is a smooth $d_j$-dimensional subset of the boundary of $X_i$ for each $j\in I_i^{d_i-1}$, and $H\Lambda^{k}(X_j)= L^2\Lambda^{k}(X_j)$. Using the extension operator from \Cref{lem: extensions}, we obtain 
	\begin{equation} \label{prod struct Hlambda}
		H \Lambda^k (X_i, \Tr) = \mathring{H}\Lambda^k(X_i) \times \Pi_{j\in I_i^{d_i-1}} \mathcal{R}\circ\phi_{i,j} (H\Lambda^k(X_j)). 
	\end{equation}
	Due to the product structure, $H \Lambda^k (X_i, \Tr)$ is therefore complete.

	For $k<d_i -1$, the same argument is applied recursively. 
\qed
\end{proof}

Next, we observe that these local spaces allow for the decomposition of $H \mathfrak{L}^k$ into a product structure similar to that derived for $L^2 \mathfrak{L}^k$ in \Cref{def: L2}. In particular, we obtain the following, alternative characterization of $H \mathfrak{L}^k$:
\begin{equation}
	\begin{aligned}\label{alt-def-H}
		H \mathfrak{L}^k(\mathfrak{F}) : \{ & \mathfrak{a} \in \mathfrak{L}^k(\mathfrak{F}) \mid 
		a_i \in H \Lambda^{k_i}(\Omega_i, \Tr)\\
		&\mathrm{ and }\ 
		a_j = \varepsilon(\Omega_{s_j},\partial_{j} \Omega_i) \Tr_{\partial_{j} \Omega_i} a_i  \quad
		\forall\ i\in I\ \mathrm{ and }\ j\in I_i 
		 \}.
	\end{aligned}
\end{equation}

\begin{lemma}\label{H-char}
	\Cref{def: HL} and \eqref{alt-def-H} are equivalent. 
\end{lemma}
\begin{proof}
	Within this proof, we denote the space defined in \Cref{def: HL} as $H_1 \mathfrak{L}^k(\mathfrak{F})$ and the space defined in equation \eqref{alt-def-H} as $H_2 \mathfrak{L}^k(\mathfrak{F})$. 

	Suppose $\mathfrak{a}\in H_1\mathfrak{L}^k(\mathfrak{F})$. 
	By the same calculation as in the proof of \Cref{lem: IBPcont}, we obtain that
	\begin{equation}
		\begin{aligned}
			(\mathfrak{a},\mathfrak{d^*b})_\mathfrak{F}  &=\\
			\sum_{i \in I} \sum_{j \in \mathfrak{S}_i} (\myd a_j + \iota_j (\Bbbd \mathfrak{a}), b_j)_{\Omega_{s_j}} 
			&+\sum_{l\in J_j^{d_j+1}}(a_j-\varepsilon(\Omega_{s_j},\partial_{j} \Omega_{s_i}) \Tr_{\partial_{j} \Omega_{s_i}} a_i, \Tr_j^* b_l)_{\Omega_{s_j}} 
		\end{aligned}
	\end{equation}
	for all $\mathfrak{b}\in \mathring{C}\mathfrak{L}^{k+1}(\mathfrak{F})$. Since the second term cannot be represented in $\mathfrak{L}^{k+1}(\mathfrak{F})$ (as it contains an inner product of forms in $\mathfrak{L}^k(\mathfrak{F})$), it must be zero for $\mathfrak{da}$ to exist. However, since $\mathfrak{b}$ is arbitrary and $C \Lambda^{k_i}(\Omega_j)$ is dense in $L^2\Lambda^{k_i}(\Omega_j)$, the continuity condition in equation \eqref{alt-def-H} is a necessary consequence of \Cref{def: HL}. Moreover, we note that as in the strong case, the weak mixed-dimensional differential can be expressed locally by the weak differentials as 
	\begin{equation}\label{weak-d-exp}
		\iota_j (\mathfrak{da}) = \myd a_j + \iota_j (\Bbbd \mathfrak{a}).
	\end{equation}
	Again, by the density of $C \Lambda^{k_i}(\Omega_j)$, it follows that the existence of $\iota_j \mathfrak{da}$ implies $a_j\in H \Lambda^{k_j}(\Omega_{s_j})$ for all $j\in \mathfrak{F}$, and therefore $H_1\mathfrak{L}^k(\mathfrak{F})\subseteq H_2\mathfrak{L}^k(\mathfrak{F})$.

	Conversely, suppose $\mathfrak{a}\in H_2\mathfrak{L}^k(\mathfrak{F})$. Then the right-hand side of equation \eqref{weak-d-exp} is well defined, and we can construct $\mathfrak{da}$. It remains to show that 
	\begin{align*}
		\|\mathfrak{da}\|_\mathfrak{d}<\infty
	\end{align*}
	However, this is a direct consequence of the definition of the norm and inner products, equations \eqref{eq:inner product} and \eqref{eq:L2norm}. Therefore also $H_2\mathfrak{L}^k(\mathfrak{F})\subseteq H_1\mathfrak{L}^k(\mathfrak{F})$, and the Lemma is proved.
\qed
\end{proof}

\begin{corollary}
	The space $C \mathfrak{L}^k$ is dense in $H\mathfrak{L}^k$ with respect to the norm \eqref{eq3.13}. 
\end{corollary}
\begin{proof}
	\Cref{H-char} shows that the space $H\mathfrak{L}^k$ is isomorphic to the product space $\Pi_{i\in I} H \Lambda^{k_i} (\Omega_{i}, \Tr)$. 
	Furthermore, equation \eqref{prod struct Hlambda} in \Cref{tr-H} shows that each local space $H \Lambda^k (\Omega_j, \Tr)$ also enjoys a product structure. Each of the factors in this product structure is isomorphic to $\mathring{H}\Lambda^{k_j}(X_j)$, for some $j$. Since $C^\infty \Lambda^{k_j}(X_j)$ is dense in these spaces, the corollary follows.
\qed
\end{proof}

A similar construction leads us to the weak codifferential operators and corresponding spaces. 

\begin{definition}\label{def:weakcodifferential} 
	For $\mathfrak{a}\in L^2\mathfrak{L}^k(\mathfrak{F})$, with $a_i \in H^*\Lambda^{k_i}(\Omega_{s_i})$ for all $i \in \mathfrak{F}$, let a weak mixed-dimensional exterior coderivative of $\mathfrak{a}$, if it exists, be any form $\mathfrak{d^*a}\in L^2\mathfrak{L}^{k-1}(\mathfrak{F})$ such that for all $\mathfrak{b}\in \mathring{H}\mathfrak{L}^{k-1}(\mathfrak{F})$,
	\begin{align} \label{int by parts weak}
		(\mathfrak{d^*a},\mathfrak{b})_{\mathfrak{F}} = 
		(\mathfrak{a}, \mathfrak{d b})_{\mathfrak{F}}.
	\end{align}
\end{definition}

\begin{definition}\label{def: H*L} 
We denote the space of weakly codifferentiable mixed-dimensional forms on the forest $\mathfrak{F}$ as
\begin{equation}
H^*\mathfrak{L}^k:\{\mathfrak{a} \in L^2\mathfrak{L}^k \mid 
\mathfrak{d}^*\mathfrak{a} \in L^2 \mathfrak{L}^{k - 1} 
\}.
\label{eq3.11-weak}
\end{equation}
We allow for boundary conditions on the external boundary $\partial_D Y$ in the sense of
\begin{equation}
\mathring{H}^*\mathfrak{L}^k:\{\mathfrak{a }\in H\mathfrak{L}^k \mid 
\Tr_{\partial_i Y}^* a_i = 0  
\text{ for all } i \in \mathfrak{F}\}.
\label{eq3.12-weak}
\end{equation}
\end{definition}
For both $H^*\mathfrak{L}^k$ and
$\mathring{H}^*\mathfrak{L}^k$ the natural norm is given as
\begin{equation}
\|\mathfrak{a}\|_{\mathfrak{d}^*} = \|\mathfrak{a}\| + \|\mathfrak{d}^*\mathfrak{a}\|.
\label{eq3.13-weak}
\end{equation}

The same considerations as elaborated above for $H\mathfrak{L}^k(\mathfrak{F})$ can be extended to $H^*\mathfrak{L}^k(\mathfrak{F})$, and we summarize these without proof: 
\begin{corollary}\label{lem: HL Hilbert2} 
The mixed-dimensional spaces $H^*\mathfrak{L}^k(\mathfrak{F})$ and $\mathring{H}^*\mathfrak{L}^k(\mathfrak{F})$ are Hilbert spaces, and the space $C^* \mathfrak{L}^k(\mathfrak{F})$ is dense in $H^*\mathfrak{L}^k(\mathfrak{F})$ with respect to the norm \eqref{eq3.13-weak}. 
\end{corollary}

\subsection{Stokes' Theorem and Poincar\'e Lemma}\label{sec3.4}

We verify that standard tools are available in the mixed-dimensional setting, starting with Stokes' theorem.
\begin{theorem}[Stokes']\label{thm: Stokes}
For any
$\mathfrak{a}\in H\mathfrak{L}^{n-1}(\mathfrak{F})$ it holds
that
\begin{equation}
\int_{\mathfrak{F}}\mathfrak{d}\mathfrak{a} = 
\sum_{i \in I} \int_{\partial_i Y} \Tr_{\partial_i Y} a_i.
\label{eq3.17}
\end{equation}
\end{theorem}
\begin{proof}
Since $\mathfrak{da} \in H\mathfrak{L}^n$ is defined only on the roots, an explicit calculation gives
\begin{align}
\int_{\mathfrak{F}}\mathfrak{da} = 
\sum_{i \in I} \int_{\Omega_i} \myd a_i + \iota_i (\Bbbd \mathfrak{a}) 
= \sum_{i \in I} 
\int_{\partial \Omega_i} \Tr_{\partial \Omega_i} a_i
 - \int_{\Omega_i} \sum_{l \in \gamma_i^{-1}} \varepsilon_{il} a_l
= \sum_{i \in I} \int_{\partial_i Y}\Tr_{\partial_i Y} a_i. \label{eq3.18}
\end{align}
\qed
\end{proof}


Stokes' theorem allows us to show that the mixed-dimensional exterior
derivative is in a certain sense the correct generalization of the
exterior derivative on $Y$ to $\Omega$.
As indicated by diagram \eqref{eq3.14}, the integral and exterior derivative
$\mathfrak{d}$ defined in \Cref{sec3.1} lead to a de Rham complex
$(H\mathfrak{L}^{\bullet},\mathfrak{d})$:
\begin{equation}
	0 \rightarrow \mathbb{R}\xrightarrow{\subset}
	H\mathfrak{L}^0\xrightarrow{\mathfrak{d}}
	H\mathfrak{L}^1\xrightarrow{\mathfrak{d}}\cdots\xrightarrow{\mathfrak{d}} 
	H\mathfrak{L}^n \rightarrow 0.
	\label{eq: De Rham H}
\end{equation}

Moreover, with boundary conditions imposed on $\partial_D Y$, \Cref{thm: Stokes} gives us the mixed-dimensional analogue of \eqref{eq2.9}
\begin{equation}
	0 \xrightarrow{\subset}
	\mathring{H} \mathfrak{L}^0\xrightarrow{\mathfrak{d}}
	\mathring{H} \mathfrak{L}^1\xrightarrow{\mathfrak{d}}\cdots\xrightarrow{\mathfrak{d}} 
	\mathring{H} \mathfrak{L}^n \xrightarrow{\int_{\mathfrak{F}}} 
	\mathbb{R} \xrightarrow{} 0.
	\label{eq: De Rham H_0}
\end{equation}

Next, we consider the cohomology classes of \eqref{eq: De Rham H} in the following theorem.

\begin{theorem} \label{Thm: Cohomology H}
	For $n \le 3$, the dimensions of the cohomology spaces of the cochain complex $(H\mathfrak{L}^{\bullet}(\mathfrak{F}),\mathfrak{d})$ are equal to those of the complex $(H \Lambda^{\bullet}(Y),\myd)$.
\end{theorem}
\begin{proof}
	The proof follows the double-complex construction of \cite{6,Weil1952}. To verify all the steps in this construction in the case of mixed-dimensional differential forms, a detailed exposition is presented in \Cref{sec: proof of cohomology}.
\qed
\end{proof}

The de Rham complex with boundary conditions
$(\mathring{H}\mathfrak{L}^{\bullet},\mathfrak{d})$
is defined similar to \eqref{eq: De Rham H}.
\begin{corollary}\label{lem: Cohomology H0}
For spaces with boundary conditions
$\mathring{H}\mathfrak{L}^{\bullet}$, all results from \Cref{Thm: Cohomology H}
also hold for the de Rham complex
$(\mathring{H}\mathfrak{L}^{\bullet},\mathfrak{d})$.
\end{corollary}

\begin{proof} 
Since the trace on $\partial Y$ commutes with the (mixed-dimensional)
exterior derivative, the result is immediate.
\qed
\end{proof}

The codifferential also defines de Rham complexes
$(H^*\mathfrak{L}^{\bullet},\mathfrak{d}^*)$ and
$(\mathring{H}^*\mathfrak{L}^{\bullet},\mathfrak{d}^*)$,
similar to \eqref{eq: De Rham C*}.

\begin{corollary}\label{lem: co-complex} 
 $(H^*\mathfrak{L}^{\bullet},\mathfrak{d}^*)$ and $(\mathring{H}^*\mathfrak{L}^{\bullet},\mathfrak{d}^*)$ are chain complexes, i.e. for all $\mathfrak{a}\in H^*\mathfrak{L}^k$,
  $\mathfrak{d}^*\mathfrak{d}^*\mathfrak{a} = 0$.
\end{corollary}

\begin{proof}

For any $\mathfrak{a}\in H^*\mathfrak{L}^k$, application of
integrating by parts twice gives:
\begin{equation}
(\mathfrak{d}^*\mathfrak{d}^*\mathfrak{a},\mathfrak{b})_{\mathfrak{F}} = 
(\Tr\mathfrak{d}\mathfrak{b},\Tr^*\mathfrak{a})_{\partial Y} + 
(\Tr\mathfrak{b},\Tr^*\mathfrak{d}^*\mathfrak{a})_{\partial Y}
\qquad\text{for all }
\mathfrak{b}\in H\mathfrak{L}^{k-2}. 
\label{eqB.2}
\end{equation}
Thus
$(\mathfrak{d}^*\mathfrak{d}^*\mathfrak{a},\mathfrak{b}) = 0$
for all $\mathfrak{a}\in\mathring{H}^*\mathfrak{L}^k$,
and it follows that $\mathfrak{d}^*\mathfrak{d}^*\mathfrak{a} = 0$
pointwise away from the boundary. But then since the boundary has no
measure it follows that
$(\mathfrak{d}^*\mathfrak{d}^*\mathfrak{a},\mathfrak{b}) = 0$
for all $\mathfrak{a}$. As a consequence, we obtain the generalization
of the integration by parts formula on the boundary
\begin{equation}
(\Tr\mathfrak{d}\mathfrak{b},\Tr^*\mathfrak{a})_{\partial Y} + 
(\Tr\mathfrak{b},\Tr^*\mathfrak{d}^*\mathfrak{a})_{\partial Y} = 0
\qquad\text{for all }\mathfrak{a}\in H^*\mathfrak{L}^k 
\text{ and }
\mathfrak{b}\in H\mathfrak{L}^{k-2}. 
\label{eqB.3}
\end{equation}
\qed
\end{proof}

\begin{remark}
	We expect the constraint $n \le 3$ in \Cref{Thm: Cohomology H} to be superfluous. However, we have not confirmed all details for $n \ge 4$. 
\end{remark}
The cohomology structure of the complexes from \Cref{lem: co-complex} will be considered in the following subsection.

\subsection{Helmholtz and Hodge decompositions}\label{sec3.3}

The Helmholtz and Hodge decompositions of the standard Sobolev spaces extend to the mixed-dimensional setting. 

\begin{theorem} \label{thm: Hodge} 
For any
$\mathfrak{a}\in L^2\mathfrak{L}^k$, there exist a Hodge
decomposition $\mathfrak{a}_{\mathfrak{d}} \in H\mathfrak{L}^k$,
$\mathfrak{a}_{\mathfrak{d}^*} \in \mathring{H}^*\mathfrak{L}^k$
and $\mathfrak{a}_0 \in \mathfrak{H}^k$ such that
\begin{equation}
\mathfrak{a} = \mathfrak{a}_{\mathfrak{d}} + \mathfrak{a}_{\mathfrak{d}^*} + \mathfrak{a}_0
\label{eq3.25}
\end{equation}
while
$\mathfrak{d}^*\mathfrak{a}_{\mathfrak{d}^*} = \mathfrak{d}\mathfrak{a}_{\mathfrak{d}} = 
\mathfrak{d}^*\mathfrak{a}_0 = \mathfrak{d}\mathfrak{a}_0 = 0$. 
The space $\mathfrak{H}^k$ is referred to as the space of
mixed-dimensional harmonic forms, is isomorphic to the cohomology, and
is defined by
\begin{equation}
\mathfrak{H}^k = \{\mathfrak{c}\in H\mathfrak{L}^k \cap 
\mathring{H}^*\mathfrak{L}^k \mid \mathfrak{d}^*\mathfrak{c} = \mathfrak{d}\mathfrak{c} = 0 \}.
\label{eq3.26}
\end{equation}
On contractible domains, $\mathfrak{H}^k = 0$ and
$\mathfrak{a}_0 = 0$, and we refer to the remaining two terms as the
Helmholtz decomposition.
\end{theorem}

\begin{proof}
We introduce the notation
$\mathfrak{d}H\mathfrak{L}^k \subset H\mathfrak{L}^{k + 1}$ and
$\mathcal{N}(H\mathfrak{L}^k,\mathfrak{d}) \subset H\mathfrak{L}^k$
to indicate the range and null-space of $\mathfrak{d}$, respectively.
Due to the integration-by-parts formula \Cref{int by parts weak} it is clear that the
space $L^2\mathfrak{L}^k$ decomposes into orthogonal complements
\begin{equation}
\mathfrak{d}H\mathfrak{L}^{k-1}\perp\mathcal{N}(\mathring{H}^*\mathfrak{L}^k,\mathfrak{d}^*),
\label{eq3.27}
\end{equation}
where
$\mathcal{N}(\mathring{H}^*\mathfrak{L}^k,\mathfrak{d}^*) \supseteq 
\mathfrak{d}^*\mathring{H}^*\mathfrak{L}^{k + 1}$.
Thus we obtain
\begin{equation}
L^2\mathfrak{L}^k = \mathfrak{d}H\mathfrak{L}^{k-1} \oplus 
\mathfrak{d}^*\mathring{H}^*\mathfrak{L}^{k + 1} \oplus 
\mathfrak{H}^k.
\label{eq3.28}
\end{equation}
Here $\mathfrak{H}^k$ is the part of $L^2\mathfrak{L}^k$
perpendicular to the first two terms. We categorize the last term: Let
$\mathfrak{c}\in\mathfrak{H}^k$, thus for all
$\mathfrak{b}\in H\mathfrak{L}^{k-1}$ and for all
$\mathfrak{a}\in \mathfrak{d}^*\mathring{H}^*\mathfrak{L}^{k + 1}$
\begin{equation}
(\mathfrak{c},\mathfrak{d}\mathfrak{b})_{\mathfrak{F}} = 0 = (\mathfrak{c},\mathfrak{d}^*\mathfrak{a})_{\mathfrak{F}}.
\label{eq3.29}
\end{equation}
Then by integration by parts, we have for all $\mathfrak{a},\mathfrak{b}$ as
above
\begin{equation}
(\mathfrak{d}^*\mathfrak{c},\mathfrak{b})_{\mathfrak{F}} + 
(\Tr\mathfrak{b},\Tr^*\mathfrak{c})_{\partial Y} = 0 = 
(\mathfrak{d}\mathfrak{c},\mathfrak{a})_{\mathfrak{F}}.
\label{eq3.30}
\end{equation}
Thus $\mathfrak{H}^k$ has the definition stated in the theorem. That
$\mathfrak{H}^k$ is isomorphic to cohomology follows directly from
the decomposition since
$\mathcal{N}(H\mathfrak{L}^k,\mathfrak{d})\mathfrak{= d} H\mathfrak{L}^{k-1} \oplus \mathfrak{H}^k$.

Finally, it follows from \Cref{Thm: Cohomology H} that $\mathfrak{H}^k$ is void if $Y$ is contractible.
\qed
\end{proof}

Note that each cohomology class contains exactly one harmonic form, and
that this minimizes the norm in its class (due to orthogonality). This
parallels the classical situation where this property is used to prove
Hodge's theorem, see e.g. \cite{14}.

\begin{lemma}\label{lem: Hodge H} 
For any
$\mathfrak{a}\in L^2 \mathfrak{L}^k$, there
exist $\mathfrak{b}_{\mathfrak{d}} \in H\mathfrak{L}^{k-1}$,
$\mathfrak{b}_{\mathfrak{d}^*} \in \mathring{H}^*\mathfrak{L}^{k + 1}$
and $\mathfrak{a}_0 \in \mathfrak{H}^k$ such that
\begin{equation}
\mathfrak{a} = 
\mathfrak{d}\mathfrak{b}_{\mathfrak{d}} + 
\mathfrak{d}^*\mathfrak{b}_{\mathfrak{d}^*} + \mathfrak{a}_0.
\label{eq3.31}
\end{equation}
\end{lemma}

\begin{proof}
Follows immediately from \Cref{thm: Hodge,Thm: Cohomology H}.
\qed
\end{proof}

The role of boundary conditions on the spaces in
\Cref{thm: Hodge} and \Cref{lem: Hodge H} can be reversed, and for completeness we
state: 

\begin{theorem} \label{thm 3.3.d} 
For any $\mathfrak{a}\in L^2\mathfrak{L}^k$, there exist a
Hodge decomposition
$\mathfrak{a}_{\mathfrak{d}} \in \mathring{H}\mathfrak{L}^k$,
$\mathfrak{a}_{\mathfrak{d}^*} \in H^*\mathfrak{L}^k$ and
$\mathfrak{a}_0 \in \mathring{\mathfrak{H}}^k$ such that
\begin{equation}
\mathfrak{a} = \mathfrak{a}_{\mathfrak{d}} + 
\mathfrak{a}_{\mathfrak{d}^*} + \mathfrak{a}_0
\label{eq3.35}
\end{equation}
while
$\mathfrak{d}^*\mathfrak{a}_{\mathfrak{d}^*} = \mathfrak{d}\mathfrak{a}_{\mathfrak{d}} = 
\mathfrak{d}^*\mathfrak{a}_0 = \mathfrak{d}\mathfrak{a}_0 = 0$. 
The space $\mathring{\mathfrak{H}}^k$ is referred to as
the space of mixed-dimensional harmonic forms with boundary conditions,
and is defined by
\begin{equation}
\mathring{\mathfrak{H}}^k = 
\{\mathfrak{c}\in\mathring{H}\mathfrak{L}^k \cap H^*\mathfrak{L}^k \mid 
\mathfrak{d}^*\mathfrak{c} = \mathfrak{d}\mathfrak{c} = 0 \}.
\label{eq3.36}
\end{equation}
The spaces $\mathring{\mathfrak{H}}^k$ are isomorphic to
the cohomology of the de Rham complex
$(\mathring{H}\mathfrak{L}^{\bullet},\mathfrak{d}^*)$.
The elements $\mathfrak{a}_{\mathfrak{d}}$ and
$\mathfrak{a}_{\mathfrak{d}^*}$ can be represented as
$\mathfrak{a}_{\mathfrak{d}} = \mathfrak{d}\mathfrak{b}_{\mathfrak{d}}$
and
$\mathfrak{a}_{\mathfrak{d}^*} = \mathfrak{b}_{\mathfrak{d}^*}$
for
$\mathfrak{b}_{\mathfrak{d}} \in \mathring{H}\mathfrak{L}^{k-1}$,
$\mathfrak{b}_{\mathfrak{d}^*} \in H\mathfrak{L}^{k + 1}$.

\end{theorem}

Two subsequent corollaries follow from the above results.
\begin{corollary} \label{lem: dimension cocohomology}
  The sequence of mappings and spaces
  $(H^* \mathfrak{L}^{\bullet}(\mathfrak{F}),\mathfrak{d}^*)$ has the
  same dimension of the cohomology space as
  $(H^* \Lambda^{\bullet}(Y),\myd^*)$.
\end{corollary}
\begin{proof}
	We note the equality of dimension of $\mathring{\mathfrak{H}}^\bullet$ and $\mathfrak{H}^\bullet$, as well as the cohomology spaces for $(H \mathfrak{L}^\bullet(\mathfrak{F}), \mathfrak{d})$, $(H \Lambda^\bullet(Y), \myd)$ and $(H^* \Lambda^\bullet(Y), \myd^*)$.
\qed
\end{proof}

\begin{corollary}
    If $Y$ is a contractible domain, then for all
    $\mathfrak{a}\in H^*\mathfrak{L}^k(\Omega)$ such
    that $\mathfrak{d}^*\mathfrak{a} = 0$, there exists
    $\mathfrak{b} \in H^*\mathfrak{L}^{k + 1}$ such that
    $\mathfrak{d}^*\mathfrak{b = a}$.
\end{corollary}
\begin{proof}
	Follows directly from the Hodge decomposition in \Cref{thm 3.3.d} combined with the voidness of $\mathring{\mathfrak{H}}^k$ for contractible domains from \Cref{lem: dimension cocohomology}.
\qed
\end{proof}

\subsection{Embeddings}\label{sec3.6}

The mixed-dimensional spaces are related to each other by compact
embeddings (abbreviated as $\subset\subset$). In particular, we will
consider the space
$\mathring{H}\mathfrak{L}^k \cap H^*\mathfrak{L}^k$, and
associate with it the norm
\begin{equation}
\|\mathfrak{a}\|_{\mathfrak{d},\mathfrak{d}^*} = 
\|\mathfrak{a}\| + \|\mathfrak{d}\mathfrak{a}\| + \|\mathfrak{d}^*\mathfrak{a}\|.
\label{eq3.37}
\end{equation}

\begin{theorem}\label{the3.4.a} 
The spaces
$H\mathfrak{L}^k \cap \mathring{H}^*\mathfrak{L}^k \subset \subset L^2\mathfrak{L}^k$
and
$\mathring{H}\mathfrak{L}^k \cap H^*\mathfrak{L}^k \subset \subset L^2\mathfrak{L}^k$.
\end{theorem}

\begin{proof}
As shorthand, we let
$G\mathfrak{L}^k = \mathring{H}\mathfrak{L}^k \cap H^*\mathfrak{L}^k$,
and likewise
$G\Lambda^k = \mathring{H}\Lambda^k \cap H^*\Lambda^k$.
We prove the theorem for this case, the other case being similar. We
recall that for Lipschitz domains it holds that
$G\Lambda^k \subset\subset L^2\Lambda^k$ \cite{15}.

It is clear from the definition of the inner products and norms that for
all $\mathfrak{a}\in G\mathfrak{L}^k$
\begin{equation}
\|\mathfrak{a}\| \lesssim \|\mathfrak{a}\|_{\mathfrak{d},\mathfrak{d}^*}.
\label{eq3.38}
\end{equation}
It remains to show precompactness, which is to say that all bounded
sequences $G\mathfrak{L}^k$ have a subsequence that is convergent in
$L^2\mathfrak{L}^k$.

Let $\mathfrak{a}_m \in G\mathfrak{L}^k$ be a bounded sequence. Since $G\mathfrak{L}^k \subset H \mathfrak{L}^k$, it follows from \Cref{H-char} that $\iota_i \mathfrak{a}_m \in H \Lambda^{k_i}(\Omega_i)$ for all nodes $i \in \mathfrak{F}$. Moreover, since $\| \iota_i (\mathfrak{d}^* \mathfrak{a}_m)\|_{\Omega_i}$ and $\| \iota_i \mathfrak{a}_m\|_{\Omega_i}$ are finite, it follows that $\| \iota_i (\myd^* \mathfrak{a}_m)\|_{\Omega_i}$ is finite as well. 
In turn, the functions $\iota_i \mathfrak{a}_m$ lie in $H\Lambda^{k_i}(\Omega_i) \cap H^*\Lambda^{k_i}(\Omega_i)$.
This allows us to construct a convergent subsequence, based on the
following argument:

\begin{enumerate}
\def\labelenumi{\arabic{enumi}.}
\item 
  We argue by induction. For a given $d$, for all $i \in I$ and $j \in I_i^d$, let $\iota_j \mathfrak{a}_m \in G\Lambda^{k_j}(\Omega_{s_j})$. 
  Then we can pass to a subsequence $\mathfrak{a}_{m'}$ of $\mathfrak{a}_m$ such that 
  $\iota_j \mathfrak{a}_{m'}$ is convergent in $L^2 \Lambda^{k_j}(\Omega_{s_j})$. We denote 
  $\overline{\mathfrak{a}}_{m'} \in G\mathfrak{L}^k$ as any smooth and bounded extension of this 
  limit sequence from \Cref{lem: extensions}.
  We now transform to a new sequence $\mathfrak{b}_m = \mathfrak{a}_{m'}-\overline{\mathfrak{a}}_{m'}$ 
  which is still bounded.
\item
  If $\mathfrak{a}_m$ satisfies the premise in step 1 for some
  $d$, the sequence $\mathfrak{b}_m$ satisfies the premise for
  $d + 1$ since clearly
  $\iota_j \mathfrak{b}_m = 0$
  for all $j \in I_i^d$, thus
  $\iota_{j'} \mathfrak{b}_m \in G\Lambda^{k_{j'}}(\Omega_{s_{j'}})$
  for all $j' \in I_i^{d + 1}$.
\item
  For any sequence $\mathfrak{a}_m$, the premise in step 1 holds for
  $d = k_i$ since then
  $H\Lambda^{k_i}(\Omega_{s_j})$ has no trace
  and
  $H\Lambda^{k_i}(\Omega_{s_j}) = \mathring{H}\Lambda^{k_i}(\Omega_{s_j})$.
\end{enumerate}

Thus any bounded sequence in $G\mathfrak{L}^k$ has a subsequence
that is convergent in $L^2\mathfrak{L}^k$, as desired.
\qed
\end{proof}

\subsection{Poincar\'e--Friedrichs Inequality}\label{sec3.5}

We close this section by extending the Poincar\'e--Friedrichs inequality to
the mixed-dimensional setting.

\begin{theorem}[Poincar\'e--Friedrichs Inequality]\label{the3.5.a} 
For
$\mathfrak{a}\in H\mathfrak{L}^k \cap \mathring{H}^*\mathfrak{L}^k$
or
$\mathfrak{a}\in \mathring{H}\mathfrak{L}^k \cap H^*\mathfrak{L}^k$,
with Hodge decompositions as in \Cref{sec3.3}, it holds that
\begin{equation}
\|\mathfrak{a}\| \le 
C(
\|\mathfrak{d}\mathfrak{a}_{\mathfrak{d}^*}\| + 
\|\mathfrak{d}^*\mathfrak{a}_{\mathfrak{d}}\|
) + 
\|\mathfrak{a}_0\|.
\label{eq3.44}
\end{equation}
\end{theorem}

\begin{proof} 
This follows directly from \Cref{the3.4.a} and standard compactness arguments.
\qed
\end{proof}

\section{Application: Elliptic partial differential equations}\label{sec4}

A large family of problems of physical interest arise as the
minimization of suitably defined energies. This allows us to derive
various problems of practical relevance, including variational problems
and strong forms of differential equations. The requisite results from
the fixed-dimensional elliptic differential equations carry over to
the mixed-dimensional setting due to the results of \Cref{sec3}, thus we will
in this section be brief in the exposition, and omit some generality and
technical details. In particular, we will unless otherwise stated
throughout this section only consider the case where $Y$ is
contractible, the general case being similar. For a background on
partial differential equations in the spirit of minimization problems
and function spaces, especially on mixed form, please confer i.e. \cite{16,7}.

We start by considering the minimization problem equivalent to a
second-order elliptic differential equation involving the
Hodge-Laplacian for
$\mathfrak{a}\in \mathring{H}\mathfrak{L}^k \cap H^*\mathfrak{L}^k$
\begin{equation}
\mathfrak{a} =\arginf_{\mathfrak{a}'\in \mathring{H}\mathfrak{L}^k \cap H^*\mathfrak{L}^k}
J_{\mathfrak{r},\mathfrak{f}}(\mathfrak{a}'),
\label{eq4.1}
\end{equation}
where we define the functional by
\begin{equation}
J_{\mathfrak{r},\mathfrak{f}}(\mathfrak{a}') = 
\tfrac12(\mathfrak{r}\mathfrak{d}^*\mathfrak{a}',\mathfrak{d}^*\mathfrak{a}') + 
\tfrac12(\mathfrak{r}^*\mathfrak{d}\mathfrak{a}',\mathfrak{d}\mathfrak{a}')-
(\mathfrak{f},\mathfrak{a}').
\label{eq4.2}
\end{equation}
Coefficients will be termed \emph{symmetric positive definite} if a scalar $r$ exists such that
\begin{equation}
\inf_{\substack{\mathfrak{b}\in\mathfrak{d}H\mathfrak{L}^k \\
\| \mathfrak{b} \|_{\Omega} = 1}}
(\mathfrak{r}\mathfrak{b},\mathfrak{b}) \geq r > 0 
\qquad\text{and}\qquad
(\mathfrak{r}\mathfrak{a},\mathfrak{b}) = (\mathfrak{a},\mathfrak{r}\mathfrak{b})
\qquad\text{for all }
\mathfrak{a},\mathfrak{b}\in H\mathfrak{L}^k. 
\label{eq4.3}
\end{equation}

\begin{theorem}\label{the4.1.a} 
For contractible domains $Y$, the functional
$J_{\mathfrak{r},\mathfrak{f}}$ has a unique minimum in
$\mathring{H}\mathfrak{L}^k \cap H^*\mathfrak{L}^k$ for
symmetric positive definite coefficients $\mathfrak{r}$ and
$\mathfrak{r}^*$.
\end{theorem}

\begin{proof}
The unique solvability of \eqref{eq4.1} is equivalent to the
coercivity of the quadratic term \cite{17}. With respect to the norm \eqref{eq3.37}, we
obtain coercivity by the Poincar\'e--Friedrichs inequality (subject to the
conditions in the proof):
\begin{equation}
\tfrac12(\mathfrak{r}\mathfrak{d}^*\mathfrak{a},\mathfrak{d}^*\mathfrak{a}) + 
\tfrac12(\mathfrak{r}^*\mathfrak{d}\mathfrak{a},\mathfrak{d}\mathfrak{a}) \geq 
\min(r,r^*)(\|\mathfrak{d}\mathfrak{a}\| + \|\mathfrak{d}^*\mathfrak{a}\|) \gtrsim 
\|\mathfrak{a}\|_{\mathfrak{d},\mathfrak{d}^*}.
\label{eq4.4}
\end{equation}
\qed
\end{proof}

From calculus of variations, we know that the minimum of \eqref{eq4.1} satisfies
the Euler--Lagrange equations when $J_{\mathfrak{r},\mathfrak{f}}$ is
differentiable. Thus
$\mathfrak{a}\in \mathring{H}\mathfrak{L}^k \cap H^*\mathfrak{L}^k$
satisfies
\begin{equation}
(\mathfrak{r}\mathfrak{d}^*\mathfrak{a},\mathfrak{d}^*\mathfrak{a}') + 
(\mathfrak{r}^*\mathfrak{d}\mathfrak{a},\mathfrak{d}\mathfrak{a}') = 
(\mathfrak{f},\mathfrak{a}')
\qquad\text{for all }
\mathfrak{a'} \in\mathring{H}\mathfrak{L}^k \cap H^*\mathfrak{L}^k.
\label{eq4.5}
\end{equation}
The existence of solutions of \eqref{eq4.5} follows directly from \Cref{the4.1.a}.
Uniqueness follows from the coercivity of the bilinear form
$(\mathfrak{r}\mathfrak{d}^*\mathfrak{a},\mathfrak{d}^*\mathfrak{a}') + 
(\mathfrak{r}^*\mathfrak{d}\mathfrak{a},\mathfrak{d}\mathfrak{a}')$.

\begin{remark}\label{rem4.1.b} 
It is important to note the structure of the
coefficient $\mathfrak{r}$ and $\mathfrak{r}^*$. Due to the
definition of the inner product \eqref{eq:inner product} and function spaces on the forest $\mathfrak{F}$, these coefficients
operate both on function values on manifolds $\Omega_i$, as
well as on the traces of functions, i.e. function values of the whole forest $\mathfrak{F}$ (but not the co-traces). These
aspects extend the concept of a ``material property'',
and form useful guidance in the design of constitutive laws for mixed-dimensional
models.
\end{remark}

\begin{remark}\label{rem4.1.x} 
In some applications, it will not be appropriate to include a differential 
equation at all intersections, and rather only consider jump conditions. 
This is reflected in a degeneracy of the coefficient for that sub-manifold. 
Thus it is of interest to consider the case where $\iota_i \mathfrak{r} \rightarrow 0$
on some manifolds. This is outside the scope of the current paper, 
but has been considered in related work \cite{13}.
\end{remark}

 \subsection{Variational problems in mixed form}\label{sec4.2}

 In order to capture explicitly physical quantities and properties of a
 problem, it is frequently preferable to consider a so-called mixed
 problem obtained by setting
 $\mathfrak{b} =-\mathfrak{r}\mathfrak{d}^*\mathfrak{a}$. Thus we consider:
 Find
 $(\mathfrak{a},\mathfrak{b}) \in H\mathfrak{L}^k \times H\mathfrak{L}^{k-1}$
 which satisfy
 \begin{alignat}{2}
 (\mathfrak{r}^{-1}\mathfrak{b},\mathfrak{b}') + (\mathfrak{a},\mathfrak{d}\mathfrak{b}') &= 0
 &&\text{for all }
 \mathfrak{b}' \in \mathfrak{L}^{k-1}, \label{eq4.6}\\
 (\mathfrak{d}\mathfrak{b},\mathfrak{a}') + 
 (\mathfrak{r}^*\mathfrak{d}\mathfrak{a},\mathfrak{d}\mathfrak{a}') &= 
 (\mathfrak{f},\mathfrak{a}')
 &\qquad&\text{for all }\mathfrak{a'} \in\mathfrak{L}^k. \label{eq4.7}
 \end{alignat}
 Equations \eqref{eq4.6}--\eqref{eq4.7} have the practical advantage that they do not
 require the coderivative. In applications, the exterior derivative often
 encompasses a conservation principle \cite{7,8,9}, and it is therefore also
 advantageous in applied computations to represent it explicitly (for an
 applied perspective, see also discussions in e.g. \cite{18,19,20,21}).

 \begin{theorem}\label{the4.2.a} 
 For uniformly bounded and non-degenerate
 coefficients $\mathfrak{r}$ and $\mathfrak{r}^*$, 
 problem \eqref{eq4.6}--\eqref{eq4.7} is well-posed.
 \end{theorem}

 \begin{proof}
 The well-posedness of the mixed formulation follows due to the
 inf-sup condition on the coupling term \cite{16,22}. Due to the Hodge decomposition
 and Poincar\'e inequality this holds by a similar calculation as in the
 fixed-dimensional case, see e.g.~the proof of Theorem 7.2 in \cite{7}.
 \qed
 \end{proof}
\begin{remark}\label{rem4.1.c} 
As announced in the introduction, the minimization problem \eqref{eq4.1} is algebraically equivalent to models used in applications, in particular for the case $k=n$. For this case, finite-dimensional approximation of equation \eqref{eq4.5} and \eqref{eq4.6}--\eqref{eq4.7} can be 
 constructed based on finite element exterior calculus \cite{7}. The resulting methods have the expected
 approximation properties based on preliminary computational examples \cite{5}. While not covered by the 
 present analysis, finite-dimensional approximations of the mixed variational problems can also be shown to allow
 for degeneracies in the coefficients which are of importance in practical computations \cite{13}.
\end{remark}

\section{Conclusions}\label{conclusions}

In this manuscript we have introduced a semi-discrete exterior
differential operator, in such a way that it corresponds to
hierarchically coupled mixed-dimensional partial differential equations.
We show that this mixed-dimensional exterior derivative inherits
standard properties from fixed-dimensional calculus, including a
codifferential, Hodge decomposition, Poincare's lemma and inequality,
and Stokes' theorem. Our approach leads to mixed-dimensional structures
of $k$-forms, similar to those of fixed-dimensional calculus, however,
the familiar Hodge duality is not available, and there is no wedge
product.

Within this setting, we define mixed-dimensional Hodge-Laplacians, and
show that they correspond to well-posed minimization problems.

As an application, we identify equations
familiar from applications. As such, the present analysis
provides a unified approach for handling problems in continuum mechanics
wherein the physical problem has high-aspect inclusion which are
beneficial to model as lower-dimensional.

\begin{acknowledgements}
The authors wish to thank Snorre Christiansen, Bernd Flemisch, Gunnar Fl{\o}ystad, 
Alessio Fumagalli, Eirik Keilegavlen, Olivier Verdier, 
Ivan Yotov, and Ludmil Zikatanov for valuable comments and discussions on this
topic.
\end{acknowledgements}

\bibliographystyle{spmpsci}
\bibliography{biblio}

\begin{thebibliography}{10}
\providecommand{\url}[1]{{#1}}
\providecommand{\urlprefix}{URL }
\expandafter\ifx\csname urlstyle\endcsname\relax
  \providecommand{\doi}[1]{DOI~\discretionary{}{}{}#1}\else
  \providecommand{\doi}{DOI~\discretionary{}{}{}\begingroup
  \urlstyle{rm}\Url}\fi

\bibitem{31}
Adams, R., Fournier, J.: Sobolev Spaces.
\newblock Pure and Applied Mathematics. Elsevier Science (2003)

\bibitem{1}
Alboin, C., Jaffr{\'e}, J., Roberts, J.E., Serres, C.: Domain decomposition for
  flow in porous media with fractures.
\newblock In: 14th Conference on Domain Decomposition Methods in Sciences and
  Engineering (1999)

\bibitem{7}
Arnold, D.N., Falk, R.S., Winther, R.: Finite element exterior calculus,
  homological techniques, and applications.
\newblock Acta Numerica \textbf{15}, 1--155 (2006)

\bibitem{22}
Babuska, I., Aziz, A.K.: Survey lectures on the mathematical foundations.
\newblock In: The Mathematical Foundations of the Finite Element Method with
  Applications to Partial Differential Equations, pp. 1--359. Academic Press,
  New York (1972)

\bibitem{2}
Bear, J.: Hydraulics of Groundwater.
\newblock McGraw-Hill (1979)

\bibitem{16}
Boffi, D., Brezzi, F., Fortin, M.: Mixed Finite Elements and Applications.
\newblock Springer, Heidelberg (2013)

\bibitem{13}
Boon, W.M., Nordbotten, J.M., Yotov, I.: Robust discretization of flow in
  fractured porous media.
\newblock SIAM Journal on Numerical Analysis \textbf{56}(4), 2203--2233 (2018)

\bibitem{6}
Bott, R., Tu, L.W.: Differential Forms in Algebraic Topology.
\newblock Springer, New York (1982)

\bibitem{17}
Braess, D.: Finite Elements: Theory, Fast Solvers and Applications in Solid
  Mechanics.
\newblock Cambridge University Press, Cambridge (2007)

\bibitem{buffa2002traces}
Buffa, A., Costabel, M., Sheen, D.: On traces for {H(curl, $\Omega$)} in
  {L}ipschitz domains.
\newblock J. Math. Anal. Appl \textbf{276}(2), 845--867 (2002)

\bibitem{30}
Buffa, A., P.~Ciarlet, J.: On traces for functional spaces related to
  {M}axwell's equations {P}art 1: {A}n integration by parts formula in
  {L}ipschitz polyhedra.
\newblock Mathematical methods in the Applied Sciences \textbf{24}, 9--30
  (2001)

\bibitem{19}
Chen, Z., Huan, G., Ma, Y.: Computational Methods for Multiphase Flows in
  Porous Media.
\newblock SIAM (2006)

\bibitem{3}
Ciarlet, P.G.: Mathematical Elasticity {V}olume {II}: Theory of Plates.
\newblock Elsevier, Amsterdam (1997)

\bibitem{24}
Formaggia, L., Fumagalli, A., Scotti, A., Ruffo, P.: A reduced model for
  {D}arcy's problem in networks of fractures.
\newblock ESAIM: Mathematical Modelling and Numerical Analysis \textbf{48}(4),
  1089--1116 (2014)

\bibitem{8}
Hiptmair, R.: Finite elements in computational electromagnetism.
\newblock Acta Numerica \textbf{11}, 237--339 (2002)

\bibitem{14}
Jost, J.: Riemannian Geometry and Geometric Analysis.
\newblock Springer, Heidelberg (1995)

\bibitem{20}
LeVeque, R.J.: Numerical Methods for Conservation Laws.
\newblock Birkh\"auser (1992)

\bibitem{11}
Licht, M.W.: Complexes of discrete distributional differential forms and their
  homology theory.
\newblock Foundations of Computational Mathematics  (2016)

\bibitem{9}
Lindell, I.V.: Differential Forms in Electromagnetics, Piscataway.
\newblock IEEE Press, NJ (2004)

\bibitem{martin2005modeling}
Martin, V., Jaffr{\'e}, J., Roberts, J.E.: Modeling fractures and barriers as
  interfaces for flow in porous media.
\newblock SIAM Journal on Scientific Computing \textbf{26}(5), 1667--1691
  (2005)

\bibitem{melrose2011remark}
Melrose, R.B.: A remark on distributions and the de {R}ham theorem.
\newblock arXiv preprint arXiv:1105.2597  (2011)

\bibitem{12}
Mitrea, D., Mitrea, M., Shaw, M.C.: Traces of differential forms on {L}ipschitz
  domains, the boundary {D}e {R}ham complex, and {H}odge decompositions.
\newblock Indiana University Mathematics Journal \textbf{57}(5), 2061--2095
  (2008)

\bibitem{21}
Nordbotten, J.M.: Convergence of a cell-centered finite volume discretization
  for linear elasticity.
\newblock SIAM Journal of Numerical Analysis \textbf{53}(6), 2605--2625 (2016)

\bibitem{5}
Nordbotten, J.M., Boon, W.M.: Modeling, structure and discretization of
  mixed-dimensional partial differential equations.
\newblock In: Domain Decomposition Methods in Science and Engineering XXIV,
  Lecture Notes in Computational Science and Engineering (2017)

\bibitem{4}
Nordbotten, J.M., Celia, M.A.: Geological Storage of CO$_2$: Modeling
  Approaches for Large-scale Simulation.
\newblock Wiley (2011)

\bibitem{15}
Picard, R.: An elementary proof for a compact imbedding result in generalized
  electromagnetic theory.
\newblock Mathematische Zeitschrift \textbf{187}, 151--164 (1984)

\bibitem{rudin2006functional}
Rudin, W.: Functional Analysis.
\newblock International series in pure and applied mathematics. McGraw-Hill
  (2006)

\bibitem{18}
Russell, T.F., Wheeler, M.F.: Finite element and finite difference methods for
  continuous flows in porous media.
\newblock In: R.E. Ewing (ed.) Mathematics of Reservoir Simulation, pp.
  35--106. SIAM (1983)

\bibitem{10}
Spivak, M.: Calculus on Manifolds.
\newblock Addison-Wesley, Reading, Massachusetts (1965)

\bibitem{Weil1952}
Weil, A.: Sur les théorèmes de de {R}ham.
\newblock Commentarii mathematici Helvetici \textbf{26}, 119--145 (1952)

\end{thebibliography}

\appendix
\section{Proof of \texorpdfstring{\Cref{Thm: Cohomology H}}{Theorem 3.30}}
\label{sec: proof of cohomology}

\begin{proof}
	From \Cref{lem: cohom Ln,lem: cohom Ln-1,lem: cohom Ln-2} below, it follows that without loss of generality, it is sufficient to consider functions $\mathfrak{a} \in H \mathfrak{L}^k$ where the only non-zero component is on $\Omega^n$. Consider therefore a form $\mathfrak{a} \in H \mathfrak{L}^k$ with $\mathfrak{da} = 0$ and $a_i = 0$ for all $i \in \mathfrak{F}$ with $d_i < n$.
	It follows that $\iota_i(\Bbbd \mathfrak{a}) = 0$ for $i \in I^{n - 1}$. In turn, the internal boundaries have no contribution on computations of integrals, which therefore depend only on the geometry of $Y$ itself. Since cohomology can be computed using integrals by de Rham's theorem, the claim follows.
\qed
\end{proof}

\begin{lemma} \label{lem: cohom Ln}
	For each $\mathfrak{a} \in H \mathfrak{L}^n$, there exists a $\mathfrak{b} \in H \mathfrak{L}^{n - 1}$ such that $\iota_i(\mathfrak{a - db}) = 0$ on all DAGs $\mathfrak{S}_i$ with $d_i < n$.
\end{lemma}
\begin{proof}
	Take $\mathfrak{a} \in H \mathfrak{L}^n$ as given and proceed according to the following three steps.

	\emph{Step A.}
		We start by considering $i \in I^0$. For any such $i$, we can construct a function $\mathfrak{c} \in H \mathfrak{L}^{n-1}$ using the extension operators from \Cref{lem: extensions}, such that
		\begin{enumerate}[label = \alph*), align = left]
			\item $\iota_i (\Bbbd \mathfrak{c}) = a_i$ 
			\item $\iota_j (\Bbbd \mathfrak{c}) = 0$ for all $j \in I^0 \setminus \{i\}$
			\item $c_{j'} = 0$ for $j' \in I$ if $d_{j'} \ne 1$ or if $I_{j'} \cap \gamma_i^{-1} = \emptyset$.
		\end{enumerate}
		
		We define the corrected function 
		\begin{align}
			\mathfrak{a' = a - dc}.
		\end{align}
		Note that $\mathfrak{a}'$ is in the same cohomology class as $\mathfrak{a}$ with $a_i'=0$, and we therefore rename $\mathfrak{a}'$ as $\mathfrak{a}$.

	\emph{Step B.}
		We repeat Step A for all $i \in I^0$.

	\emph{Step C.}
	Steps A and B are repeated for all $i \in I$ with $0 < d_i < n$ in increasing order of dimension. For each $i$, the correction function $\mathfrak{c} \in H \mathfrak{L}^{n-1}$ then satisfies
		\begin{enumerate}[label = \alph*), align = left]
			\item $\iota_i (\Bbbd \mathfrak{c}) = a_i$ 
			\item $\iota_j (\Bbbd \mathfrak{c}) = 0$ for all $j \in I^{d_i} \setminus \{i\}$
			\item $c_{j'} = 0$ for $j' \in I$ if $d_{j'} \ne d_i + 1$ or if $I_{j'} \cap \gamma_i^{-1} = \emptyset$.
		\end{enumerate}

	Defining $\mathfrak{b}$ as the sum of all consecutive correction functions, we arrive at the claim.
\qed
\end{proof}

\begin{lemma} \label{lem: cohom Ln-1}
	Given $n \le 3$, for each $\mathfrak{a} \in H \mathfrak{L}^{n - 1}$ with $\mathfrak{da} = 0$, there exists a $\mathfrak{b} \in H \mathfrak{L}^{n - 2}$ such that $\iota_i(\mathfrak{a - db}) = 0$ on all DAGs $\mathfrak{S}_i$ with $d_i < n$.
\end{lemma}
\begin{proof}
	Take $\mathfrak{a} \in H \mathfrak{L}^{n - 1}$ as given, then by definition $a_i = 0$ for all $i \in I^0$. We nonetheless first introduce a correction according to each zero-dimensional manifold such that the corresponding traces of $\mathfrak{a}$ vanish. We continue according to the following steps.

	\emph{Step A.}
		Consider any fixed $i \in I^0$. By continuity, each $a_j$ with $j \in \gamma_i^{-1}$ represents a neighboring trace of $\mathfrak{a}$ from the 1-manifolds adjacent to $i$.	We then introduce $\mathfrak{c} \in H \mathfrak{L}^{n - 2}$, using the extension operators from \Cref{lem: extensions}, such that 
		\begin{enumerate}[label = \alph*), align = left]
			\item $\iota_j(\Bbbd \mathfrak{c}) = a_j$ for each $j \in \gamma_i^{-1}$
			\item $\iota_{j'} (\Bbbd \mathfrak{c}) = 0$ for all $j' \in I_l^0 \setminus \gamma_i^{-1}$ with $l \in I^1$
			\item $c_{j''} = 0$ for ${j''} \in I$ if $d_{j''} \ne 2$ or if $I_{j''} \cap \gamma_i^{-2} = \emptyset$.
		\end{enumerate}

		Constraint a) above involves solving a system of $|\gamma_i^{-1}|$ equations with $|\gamma_i^{-2}| = | \bigcup_{j \in \gamma_i^{-1}} \gamma_j^{-1}|$ unknowns. This system is solvable by the fact that $\iota_i (\Bbbd \mathfrak{a}) = 0$ and the assumed geometry.

		We define the corrected function $\mathfrak{a' = a - dc}$. Note that $\mathfrak{a}'$ is in the same cohomology class as $\mathfrak{a}$ with $a_j'=0$, and we therefore rename $\mathfrak{a}'$ as $\mathfrak{a}$.

	\emph{Step B.}
		We repeat Step A for all $i \in I^0$.

	\emph{Step C.}
		Consider any fixed $i \in I^1$. We construct a function $\mathfrak{c} \in H \mathfrak{L}^{n - 2}$ such that
		\begin{enumerate}[label = \alph*), align = left]
			\item $\iota_i (\Bbbd \mathfrak{c}) = a_i$ 
			\item $\iota_j (\Bbbd \mathfrak{c}) = 0$ for all $j \in I^1 \setminus \{i\}$
			\item $c_j = 0$ for $j \in I$ if $d_j \ne 2$ or if $I_j \cap \gamma_i^{-1} = \emptyset$.
		\end{enumerate}

		Introducing $\mathfrak{a' = a - dc}$, we note that $a_i' = 0$ and we rename $\mathfrak{a}'$ as $\mathfrak{a}$.

	\emph{Step D.}
		Step C is repeated for all $i \in I^1$. The result is now shown for $n = 2$.

	\emph{Step E.}
		Continuing with $n = 3$, consider any fixed $i \in I^2$. We construct a function $\mathfrak{c} \in H \mathfrak{L}^{n - 2}$, using the extension operators from \Cref{lem: extensions}, such that
		\begin{enumerate}[label = \alph*), align = left]
			\item $\iota_i (\Bbbd \mathfrak{c}) = a_i$ 
			\item $\iota_j (\Bbbd \mathfrak{c}) = 0$ for all $j \in I^2 \setminus \{i\}$
			\item $c_j = 0$ for $j \in I$ if $d_j \ne 3$ or if $I_j \cap \gamma_i^{-1} = \emptyset$.
		\end{enumerate}
		Introducing $\mathfrak{a' = a - dc}$, we note that $a_i' = 0$ and we rename $\mathfrak{a}'$ as $\mathfrak{a}$.

	\emph{Step F.}
		Repeating Step E for all $i \in I^2$ proves the result for $n = 3$.
\qed
\end{proof}

\begin{lemma} \label{lem: cohom Ln-2}
	Given $n = 3$, for each $\mathfrak{a} \in H \mathfrak{L}^1$ with $\mathfrak{da} = 0$, there exists a $\mathfrak{b} \in H \mathfrak{L}^0$ such that $\iota_i(\mathfrak{a - db}) = 0$ on all DAGs $\mathfrak{S}_i$ with $d_i < n$.
\end{lemma}
\begin{proof}

	As in the previous lemmas, we correct a given form $\mathfrak{a}$. Note that here, we only need to correct $\mathfrak{a}$ on DAGs $\mathfrak{S}_i$ with $d_i = 2$. We start with the zero-dimensional nodes in these DAGs and progress in increasing order of dimension.
	
	\emph{Step A.} 	Consider $i \in I^0$ fixed.
	We construct $\mathfrak{c} \in H \mathfrak{L}^0$ such that 
	\begin{enumerate}[label = \alph*), align = left]
		\item $\iota_j(\Bbbd \mathfrak{c}) = a_j$ for each $j \in \gamma_i^{-2}$
		\item $\iota_{j'} (\Bbbd \mathfrak{c}) = 0$ for all $j' \in I_l^0 \setminus \gamma_i^{-2}$ with $l \in I^2$
		\item $c_{j''} = 0$ for $j'' \in I$ if $I_{j''} \cap \gamma_i^{-3} = \emptyset$.
	\end{enumerate}
	In the above, a) involves solving a system of $|\gamma_i^{-2}|$ equations with $|\gamma_i^{-3}| = | \bigcup_{j \in \gamma_i^{-2}} \gamma_j^{-1}|$ unknowns. Due to the fact that $\iota_{j'} (\Bbbd \mathfrak{a}) = 0$ for all $j' \in \gamma_i^{-1}$ this system reduces to $|\gamma_i^{-2}| - |\gamma_i^{-1}|$ equations, which is solvable by the geometric constraints.

	We define the corrected function $\mathfrak{a' = a - dc}$. Note that $\mathfrak{a}'$ is in the same cohomology class as $\mathfrak{a}$ with $a_j'=0$, and we therefore rename $\mathfrak{a}'$ as $\mathfrak{a}$.

	\emph{Step B.} Repeat step A for all $i \in I^0$.

	\emph{Step C.} Consider $i \in I^1$ fixed. Let $j \in \gamma_i^{-1}$ and consider $a_j$. We construct $\mathfrak{c} \in H \mathfrak{L}^0$ such that 
	\begin{enumerate}[label = \alph*), align = left]
		\item $\iota_j(\Bbbd \mathfrak{c}) = a_j$ for each $j \in \gamma_i^{-1}$
		\item $\iota_{j'} (\Bbbd \mathfrak{c}) = 0$ for all $j' \in I_l^1 \setminus \gamma_i^{-1}$ with $l \in I^2$
		\item $c_{j''} = 0$ for $j'' \in I$ if $I_{j''} \cap \gamma_i^{-2} = \emptyset$.
	\end{enumerate}

	In the above, a) involves solving a system of $|\gamma_i^{-1}|$ equations with $|\gamma_i^{-2}|$ unknowns. Due to the fact that $\iota_i (\Bbbd \mathfrak{a}) = 0$ this system is solvable.
	Introducing $\mathfrak{a' = a - dc}$, we note that $a_j' = 0$ and we rename $\mathfrak{a}'$ as $\mathfrak{a}$.

	\emph{Step D.} Repeat step C for all $i \in I^1$.

	\emph{Step E.} Consider $i \in I^2$ fixed.
	We construct $\mathfrak{c} \in H \mathfrak{L}^0$ such that 
	\begin{enumerate}[label = \alph*), align = left]
		\item $\iota_i(\Bbbd \mathfrak{c}) = a_i$
		\item $\iota_j (\Bbbd \mathfrak{c}) = 0$ for all $j \in I^2 \setminus \{i\}$
		\item $c_{j'} = 0$ for $j \in I$ if $I_{j} \cap \gamma_i^{-1} = \emptyset$.
	\end{enumerate}
	Introducing $\mathfrak{a' = a - dc}$, we note that $a_i' = 0$ and we rename $\mathfrak{a}'$ as $\mathfrak{a}$.

	\emph{Step F.} Repeat step E for all $i \in I^2$ to obtain the result.
\qed
\end{proof}

\end{document}